\documentclass[journal,12pt,onecolumn,draftclsnofoot]{IEEEtran}

\usepackage{amsmath}
\usepackage{amssymb}
\usepackage{amsthm}
\usepackage{mathrsfs, color}
\usepackage{hyperref}
\usepackage{fullpage}
\usepackage{subfigure}
\usepackage{graphicx}
\usepackage{graphics, color}
\usepackage{multirow}
\usepackage{url}
\usepackage{enumitem}

\newcommand{\va}[0]{\mathbf a}
\newcommand{\vb}[0]{\mathbf b}

\newcommand{\ve}[0]{\mathbf e}

\newcommand{\vs}[0]{\mathbf s}
\newcommand{\vu}[0]{\mathbf u}

\newcommand{\vw}[0]{\mathbf w}
\newcommand{\vx}[0]{\mathbf x}
\newcommand{\vy}[0]{\mathbf y}
\newcommand{\vz}[0]{\mathbf z}

\newcommand{\vX}[0]{\mathbf X}

\newcommand{\veta}[0]{\boldsymbol\eta}

\newcommand{\vzero}[0]{\mathbf 0}
\newcommand{\vone}[0]{\mathbf 1}
\newcommand{\mbf}[1]{\mbox{\boldmath$#1$}}

\newcommand{\Var}[0]{\text{Var}}

\newcommand{\tr}[0]{\text{tr}}

\newcommand{\calF}[0]{\mathcal{F}}
\newcommand{\calG}[0]{\mathcal{G}}
\newcommand{\calP}[0]{\mathcal{P}}

\newcommand{\E}[0]{\mathbb{E}}
\newcommand{\Id}[0]{\text{Id}}
\renewcommand{\span}[0]{\text{span}}
\newcommand{\vxi}[0]{\mbf \xi}
\newcommand{\sign}[0]{\text{sign}}

\newcommand{\rbr}[1]{\left(#1\right)}
\newcommand{\rbR}[1]{\left[#1\right]}
\newcommand{\rBr}[1]{\left\{#1\right\}}
\newcommand{\rBR}[1]{\left|#1\right|}

\newcommand{\Prob}[0]{\mathbb{P}}
\newcommand{\I}[0]{\mathbb{I}}

\newtheorem{thm}{Theorem}[section]
\newtheorem{lem}[thm]{Lemma}
\newtheorem{prop}[thm]{Proposition}

\theoremstyle{definition}
\newtheorem{defn}{Definition}[section]

\theoremstyle{remark}
\newtheorem{rmk}{Remark}

\newcommand{\Z}[0]{\mathbb{Z}}
\newcommand{\cP}[0]{{\cal P}}

\ifCLASSINFOpdf
\else
\fi
\begin{document}
%
\title{Regularized estimation of linear functionals of precision matrices for high-dimensional time series}
%
%
%
\author{Xiaohui~Chen,~
        Mengyu~Xu,~
        and~Wei Biao~Wu~
\thanks{X. Chen is with the Department
of Statistics, University of Illinois at Urbana-Champaign,
IL, 61820 USA. E-mail: (xhchen@illinois.edu).}
\thanks{M. Xu is with the Department of Statistics, University of Central Florida, FL, 32816 USA. E-mail: (Mengyu.Xu@ucf.edu).}
\thanks{W. B. Wu is with the Department of Statistics, University of Chicago, IL, 60637 USA. E-mail: (wbwu@galton.uchicago.edu).}}
\maketitle

\begin{abstract}
This paper studies a Dantzig-selector type regularized estimator for linear functionals of high-dimensional linear processes. Explicit rates of convergence of the proposed estimator are obtained and they cover the broad regime from i.i.d. samples to long-range dependent time series and from sub-Gaussian innovations to those with mild polynomial moments. It is shown that the convergence rates depend on the degree of temporal dependence and the moment conditions of the underlying linear processes. The Dantzig-selector estimator is applied to the sparse Markowitz portfolio allocation and the optimal linear prediction for time series, in which the ratio consistency when compared with an oracle estimator is established. The effect of dependence and innovation moment conditions is further illustrated in the simulation study. Finally, the regularized estimator is applied to classify the cognitive states on a real fMRI dataset and to portfolio optimization on a financial dataset.
\end{abstract}


%
\IEEEpeerreviewmaketitle

%
%
%
%

\section{Introduction}
\label{sec:introduction}

\IEEEPARstart{M}{ultivariate} time series data arise in a broad spectrum of real applications. Let $\vx_i, i \in \Z$, be a $p$-dimensional stationary time series with mean $\mbf\mu$ and covariance matrix $\Sigma = {\rm cov}(\vx_i)$. Given the sample $\vx_i, i = 1, \ldots, n$, we consider estimation of linear functionals of the form $\mbf\theta = \Sigma^{-1}\vb$ where $\vb$ is a $p \times 1$ vector. Such functionals appear in Markowitz Portfolio (MP) allocation, linear discriminant analysis (LDA), beamforming in array signal processing, best linear unbiased estimator (BLUE) and optimal linear prediction for univariate time series. See \cite{markowitz1952,anderson2003,guerci1999a,abrahamssonselenstoica2007}, which all can be formulated as solutions of the general linear equality constrained quadratic programming (QP) problem
\begin{equation}
\label{eqn:constrianed_opt}
\text{minimize}_{\vw \in \mathbb{R}^p} \quad \vw^\top \Sigma \vw \quad  \text{subject to} \quad \vw^\top \vb = m.
\end{equation}
It is clear that the solution is $\vw^* = m \Sigma^{-1} \vb / (\vb^\top \Sigma^{-1} \vb) \propto \mbf\theta$ and value of (\ref{eqn:constrianed_opt}) is $m^2/(\vb^\top \Sigma^{-1} \vb)$.

To estimate $\mbf\theta$, traditional approaches take two steps: (i) an estimate $\hat\Sigma$ of $\Sigma$ is constructed and (ii) estimate $\mbf\theta$ using $\hat\Sigma^{-1} \vb$ or $\hat\Sigma^{-1} \hat\vb$ if $\vb$ is unobserved. Although the two-step estimator is asymptotically consistent for $\mbf\theta$ in the classical fixed and low dimensional case, it may no longer work in high dimensions. First, consistent estimation of $\Sigma$ or its inverse is a challenging problem in the high-dimensional setting. Under sparseness or other structural conditions on $\Sigma$ or $\Sigma^{-1}$, researchers studied regularized covariance matrix estimators \cite{karoui2008,bickellevina2008a, bickellevina2008b, fanfengwu2009a,chenwieseleldarhero2010a, chenkimwang2012, hoffbecklandgrebe1996, stocializhuguerci2008, marjanovichero2015, chenwangmckeown2012b, wieselbibibloberson2013, bergejensensolberg2007a}. Without such structural conditions it is unclear how one can obtain a consistent estimator. Second, consistent estimation of $\Sigma$ or its inverse does not automatically imply consistency of $\hat \Sigma^{-1} \vb$ or $\hat\Sigma^{-1} \hat\vb$ since $|\vb|_2=\sqrt{\vb^\top \vb}$ may also increase with the dimension $p$. Indeed, to estimate $\mbf\theta$ by a "plug-in" method $\hat{\mbf\theta}=\hat\Sigma^{-1}\vb$, we can only get in the worst case $|\hat{\mbf\theta}-\mbf\theta|_2 = |\hat\Sigma^{-1} \vb - \Sigma^{-1} \vb|_2 \le \rho(\hat\Sigma^{-1}-\Sigma^{-1}) |\vb|_2$, where $\rho$ is the spectral norm. If $|\vb|_2$ diverges to infinity at a faster rate than $\rho(\hat\Sigma^{-1}-\Sigma^{-1})$, then the plug-in estimator does not converge.

Direct estimation for functionals of covariance matrices is studied in \cite{shaowangdengwang2011a, bickellevina2004, maizouyuan2012a, kolarliu2013a, cailiu2011d, cailiuluo2011a} among others for independent and identically distributed (i.i.d.) data. Allowing serial dependence, \cite{chenxuwu2013a} established an asymptotic theory for sparse covariance matrix estimators. That work, however, does not directly deal with estimating the linear functional $\mbf\theta$ and it can only handle weakly temporal dependent processes. It rules out many interesting applications such as long memory or long-range dependent time series in the fields of hydrology, network traffic, economics and finance (\cite{beranfengghoshkulik2013, taqquwillingersherman1997, cont2005long, hurst1950long}).

In this paper we shall focus on direct estimation of $\mbf \theta$ for both short- and long-range dependent times series. Here we assume that $(\vx_i)$ has the form of vector linear process
\begin{equation}
\label{eqn:linproc}
\vx_i = \mbf\mu + \sum_{m=0}^\infty A_m \vxi_{i-m},
\end{equation}
where $\mbf\mu$ is the mean vector, $A_m$ are $p \times p$ coefficient matrices, $\vxi_i = (\xi_{i,1}, \cdots, \xi_{i,p})^\top$, and $(\xi_{i,j})_{i,j \in \mathbb{Z}}$ are i.i.d. random variables (a.k.a. innovations) with zero mean and unit variance. To develop high-dimensional asymptotics, following the setting in Section 2.4.2 in \cite{MR2807761}, we shall deal with the triangular array of observations of $p_k$-dimensional vectors $(\vx^{(k)}_{i})$, $i = 1, \ldots, n_k$, $k = 1, 2, \ldots$, with $\min(n_k, p_k) \to \infty$.  Hereafter for notational simplicity, we omit the subscript $k$ and the asymptotic relation is referred to as $\min(n, p) \to \infty$.

Vector linear process is a flexible model in that $A_m$ captures both the spatial and temporal dependences. The decay rate of $A_m$ (see (\ref{eqn:coefs_decay})) is associated to \emph{temporally weakly and temporally strongly} dependent, both of which we shall deal with. An important special case of (\ref{eqn:linproc}) is the stationary Gaussian process. Another example is the vector auto-regression (VAR) model
\begin{equation}
\label{eqn:var1}
\vx_i = B_1 \vx_{i-1} + \ldots + B_d \vx_{i-d}  + \mbf\xi_i,
\end{equation}
where $B_1, \ldots, B_d$ are coefficient matrices such that (\ref{eqn:var1}) has a stationary solution. The above model is widely used in economics and finance \cite{stock2001vector, fan2011sparse, ledoit2003improved, tsay2005analysis,sim1980}. Recent developments have been made in the estimation and sparse recovery of the VAR model under high dimensionality \cite{MR2427370, MR2755014, MR2301500, 1307.0293, 1311.4175}. The linear process model (\ref{eqn:linproc}) is quite flexible to include: (i) long-range dependence (LRD); (ii) non-Gaussian distributions with possibly heavy-tails. In the network traffic analysis \cite{taqquwillingersherman1997}, it is well-recognized that: (i) is the \emph{Joseph effect}, i.e. the degree of self-similarity; and (ii) is the \emph{Noah effect}, i.e. the heaviness of the tail. In addition, those concerns are also amenable to a large body of other real applications in financial, economic, as well as biomedical engineering such as the functional Magnetic Resonance Imaging (fMRI) and microarray data \cite{dinovetal2005a, posekanyfelsensteinsykacek2011} where the signal-to-noise ratio can be low.

\subsection{Method and key assumptions}
\label{subsec:setup-assumptions}

We propose the following Dantzig-type \cite{candestao2007a, cailiuluo2011a} estimator
\begin{equation}
\label{eqn:functional}
\hat{\mbf\theta} := \hat{\mbf\theta}(\lambda) = \text{argmin}_{\mbf\eta \in {\mathbb R}^p} \left\{ |\mbf\eta|_1 : | \hat{S}_n \mbf\eta - \hat\vb |_\infty \le \lambda \right\},
\end{equation}
where $\hat\vb$ is an estimator of $\vb$ and $\hat{S}_n$ is the sample covariance matrix. If $\vb$ is known, then we can simply use $\hat\vb=\vb$. If the mean vector is known, $\hat{S}_n=n^{-1}\sum_{i=1}^n (\vx_i-\mbf\mu) (\vx_i-\mbf\mu)^\top$; otherwise $\hat{S}_n=n^{-1} \sum_{i=1}^n (\vx_i-\bar{\vx})(\vx_i-\bar{\vx})^\top$. Compared with the two-step methods, the estimate $\hat{\mbf\theta}$ in (\ref{eqn:functional}) has two advantages in terms of both theory and computation. First, since $\mbf\theta$ is a $p \times 1$ vector, there are only $p$ parameters to estimate. Rate of convergence for $\hat{\mbf\theta}$ in (\ref{eqn:functional}) can be obtained under very general temporal dependence and mild moment conditions; see Theorems \ref{thm:linear_stat_rate_bounded}--\ref{thm:linear_stat_rate_polynomial} in Section \ref{sec:rates_linear_statistics}. Second, $\hat{\mbf\theta}$ can be recast as an augmented linear program (LP)
\begin{eqnarray*}
\text{minimize}_{\vu \in {\mathbb R}^p_+, {\mbf\eta} \in {\mathbb R}^p} & & \sum_{j=1}^p u_j \\
\text{subject to} & & -\eta_j \le u_j, \;\;\; \eta_j \le u_j, \;\; \forall j = 1, \cdots, p, \\
& & -\hat{\vs}_k^\top \mbf\eta + \hat{b}_k \le \lambda, \\
& & \hat{\vs}_k^\top \mbf\eta - \hat{b}_k \le \lambda, \;\; \forall k = 1, \cdots, p,
\end{eqnarray*}
where $\hat{\vs}_k$ is the $k$-th column of $\hat{S}_n$. Let $(\hat\vu, \hat{\mbf\eta})$ be a solution of the LP; then $\hat{\mbf\theta} = \hat{\mbf\eta}$. There are computationally efficient off-the-shelf LP solvers to obtain numerical values of $\hat{\mbf\theta}$ for large-scale problems. Our estimate and the equivalent LP is similar to the CLIME estimate \cite{cailiuluo2011a}, where $\hat\vb$ is chosen to be the fixed Euclidean basis vectors.

Now, we state our key assumptions and discuss their implications. First, we need to impose conditions on the temporal dependence. Write $A_m = (a_{m, jk})_{1 \le j,k \le p}$; let $C_0 \in (0, \infty)$ be a finite constant. We assume that the linear process satisfies the decay condition
\begin{equation}
\label{eqn:coefs_decay}
\max_{1 \le j \le p}  |A_{m,j\cdot}| = \max_{1 \le j \le p}  ( \sum_{k=1}^p a_{m, jk}^2 )^{1/2} 
 \le C_0 (1 \vee m)^{-\beta}
\end{equation}
for all $m \ge 0$, where $\beta > 1/2$ and $|A_{m,j\cdot}|$ is the $\ell^2$ norm of the $j$-th row of $A_m$. If $\beta > 1$, (\ref{eqn:coefs_decay}) implies short-range dependence (SRD) since the auto-covariance matrices $\Sigma_k = \sum_{m=0}^\infty A_m A_{m+k}^\top$ are absolutely summable. On the other hand, if $1 > \beta > 1/2$, then $(\vx_i)$ in (\ref{eqn:linproc}) may not have summable auto-covariance matrices, thus allowing long-range dependence (LRD). The classical literature on LRD primarily focuses on the univariate case $p = 1$.

Next, we shall specify the tail conditions on the innovations $\xi_{i, j}$. We say that $\xi_{i, j}$ is sub-Gaussian if there exists $t > 0$ such that $\E \exp(t \xi_{1,1}^2) < \infty$, or equivalently, there exists a constant $C_\xi < \infty$ such that
\begin{eqnarray}
\label{eq:J280946p}
\| \xi_{1,1} \|_q := [\E(|\xi_{1,1}|^q)]^{1/q} \le C_\xi q^{1/2}
\end{eqnarray}
holds for all $q \ge 1$. A slightly weaker version is the (generalized) sub-exponential distribution. Let $\alpha > 1/2$. Assume that for some $t > 0$, $\E \exp(t |\xi_{1,1}|^{1/\alpha}) < \infty$, or 
\begin{eqnarray}
\label{eq:J290945}
\| \xi_{1,1} \|_q \le C_{\xi,\alpha} q^\alpha \mbox{ holds for all } q \ge 1. 
\end{eqnarray}
Equivalently, for all $x \ge 0$, $\Prob(|\xi_{1,1} | \ge x) \le C_1 \exp(-C_2 x^{1/\alpha})$ holds for some $C_1, C_2 > 0$. In the study of vector autoregressive processes, the issue of fat tails can frequently arise \cite{sim1980} and it can affect the validity of the associated statistical inference. In this paper we shall also consider the case in which $\xi_{i, j}$ only has finite polynomial moment: there exists a $q \ge 1$ such that 
\begin{eqnarray}
\label{eq:J291004}
\| \xi_{1,1} \|_q  < \infty. 
\end{eqnarray}
The tail distribution condition (or equivalently the moment condition) severely affects rates of convergence of various covariance matrix estimates. As a  primary goal of this paper, we shall develop an asymptotic theory for convergence rates of linear functional estimates with various levels of temporal dependence and for innovations having sub-Gaussian (including bounded and Gaussian as special cases) (cf (\ref{eq:J280946p})), sub-exponential (cf. (\ref{eq:J290945})) and algebraic (cf. (\ref{eq:J291004})) tails.

Finally, we assume that the linear functional $\mbf\theta$ is ``sparse" in the sense that most of its entries have small magnitudes. This is a plausible assumption in real applications such as portfolio selection \cite{brodieetal2009,fanzhangyu2012a}, LDA \cite{maizouyuan2012a}, optimal estimation and prediction for time series \cite{chen2015}. For instance, to obtain stable portfolio optimization and facilitate transaction costs for a large number of assets,  \cite{brodieetal2009} considered sparse portfolio by adding an $\ell^1$ penalty in the objective function. In LDA, classification based on the sparse Bayes direction has been studied in \cite{maizouyuan2012a}. 
Our estimator (\ref{eqn:functional}) is also closely related to the Dantzig selector for the linear regression model \cite{candestao2007a}. Let $\vy = \vX \mbf\theta + \ve$, where $\vX^\top = n^{-1/2} (\vx_1, \cdots, \vx_n)$ is the design matrix and $\ve \sim N(\vzero, \Id_{n \times n})$. The Dantzig selector is defined as the solution of
\begin{equation}
\label{eqn:DS}
\text{minimize}_{\mbf\eta \in {\mathbb R}^p} |\mbf\eta|_1 \quad \text{subject to} \quad | \vX^\top (\vX \mbf\eta - \vy) |_\infty \le \lambda.
\end{equation}
Since $\vX^\top (\vX \mbf\eta - \vy) = \hat{S}_n \mbf\eta - \vX^\top \vy$, (\ref{eqn:DS}) is equivalent to (\ref{eqn:functional}) with $\hat\vb = \vX^\top \vy$. When the dimension $p$ is large, it is reasonable to assume that prediction using a small number of predictors is desirable for practical modeling, statistical analysis and interpretation.

\section{Main results}
\label{sec:methods}

In this section, we shall first present the rate of convergence of (\ref{eqn:functional}) for the linear functional $\mbf\theta = \Sigma^{-1} \vb$. The convergence rate is characterized under various vector norms for linear processes with a broad range of dependence levels and tail conditions. Then, we present two applications to derive the ratio consistency of direct estimation for sparse Markowitz portfolio allocation and optimal linear prediction.

We now introduce some notation. Denote by $C, C', C_1, C_2, \cdots$ positive constants (independent of the sample size $n$ and the dimension $p$), whose values may vary from place to place. Let $\va$ be a vector in $\mathbb{R}^p$, $M$ be a $p \times p$ matrix, $X$ be a random variable and $q > 1$. Write $|\va|_q = (\sum_{j=1}^p |a_j|^q)^{1/q}$, $|\va| = |\va|_2$ and $|\va|_\infty = \max_{1\le j \le p} |a_j|$. Let $\rho(M) = \max\{|M \va| : |\va| = 1\}$ be the spectral norm of $M$, $|M|_{L^1} = \max_{1 \le k \le p} \sum_{j=1}^p |m_{jk}|$, $|M|_F = (\sum_{j,k=1}^p m_{jk}^2)^{1/2}$ and $|M|_1 = \sum_{j,k=1}^p |m_{jk}|$. We write $X \in {\cal L}^q$ if $\|X\|_q = (\E|X|^q)^{1/q} < \infty$. Denote $\|X\| = \|X\|_2$. For two sequences of quantities $a := a_{n,p}$ and $b := b_{n,p}$, we use $a \lesssim b$, $a \asymp b$, $a \sim b$ and $a \ll b$ to denote $a \le C_1 b$, $C_2 b \le a \le C_3 b$, $a/b \to 1$ and $a/b \to 0$ as $p,n \to \infty$, respectively. We use $a \wedge b = \min(a, b)$, $a \vee b = \max(a, b)$, $a_+ = \max(a, 0)$ and $\sign(a) = 1, 0, -1$ if $a > 0$, $a = 0$ and $a < 0$, respectively. For a set $\cal S$, $|{\cal S}|$ is the cardinality of $\cal S$. Throughout the paper, we use $\beta'=\min(2\beta-1, 1/2)$.

\subsection{Convergence rates for estimating linear functionals}
\label{sec:rates_linear_statistics}

Without loss of generality, we assume $\mbf\mu=\mbf 0$.  We shall use the smallness measure 
\begin{eqnarray*}
D(u) = \sum_{j=1}^p (|\theta_j| \wedge u), u \ge 0,
\end{eqnarray*}
to quantify the size of $\mbf\theta$. Let $0 \le r < 1$ and
\begin{equation*}
\calG_r(\nu, M_p) = \left\{ \mbf\eta \in \mathbb{R}^p : \max_{j \le p} |\eta_j| \le \nu, \; \sum_{j=1}^p |\eta_j|^r \le M_p \right\},
\end{equation*}
which contains approximately sparse vectors in the strong $\ell^r$-ball. Here, $\nu$ is a constant independent of $p$ and we allow $M_p$ to grow with $p$. If $\mbf\theta \in \calG_r(\nu, M_p)$, then $D(u) \le C_{r, \nu} M_p u^{1-r}$. Suppose that $r_b$ is the rate of $\hat\vb$ for estimating $\vb$ such that
\begin{equation}
\label{eqn:mean-rate}
\Prob(|\hat\vb-\vb|_\infty > c_b r_b) \le 2 p^{-C_b} 
\end{equation}
for some constants $c_b,C_b>0$. If $\vb$ is observed, we can take $r_b=0$ and $C_b=\infty$.

\begin{thm}[Sub-Gaussian]
\label{thm:linear_stat_rate_bounded}
Let $(\vx_i)$ be the linear process defined in (\ref{eqn:linproc}) that satisfies (\ref{eqn:coefs_decay}) and (\ref{eq:J280946p}). Let  $J_{n,p,\beta} = (\log(p) / n)^{1/2}$, $(\log(p) / n)^{1/2} \vee (\log(p) / n^{2\beta-1})$, and $\log(p) / n^{2\beta-1}$, for $\beta > 1$, $1 > \beta > 3/4$ and $3/4 > \beta > 1/2$, respectively. Let $c_b$ and $C_b $ be constants defined in (\ref{eqn:mean-rate}). Then there exist constants $C_1>1, C_2>0$ only depending on $\beta$, $C_0$ in (\ref{eqn:coefs_decay}) and $C_\xi$ in (\ref{eq:J280946p}), such that for $\lambda \ge c_b r_b+ C_1 |\mbf\theta|_1 J_{n,p,\beta}$, with probability at least $1-2p^{-C_b}-2p^{-C_2}$ we have
\begin{equation}
\label{eqn:linear_stat_wnorm_rate_bounded}
|\hat{\mbf\theta}- \mbf\theta|_w \le [6D(5|\Sigma^{-1}|_{L^1} \lambda)]^{1 \over w} (2 |\Sigma^{-1}|_{L^1}\lambda)^{1- {1 \over w}}
\end{equation}
for $1 \le w \le \infty$. In particular, for $\mbf\theta \in \calG_r(\nu, M_p)$, with the choice $\lambda = c_br_b+ C_1\nu^{1-r} M_p J_{n,p,\beta}$, we have
\begin{eqnarray}
\nonumber
 \Prob\left( |\hat{\mbf\theta} - \mbf\theta|_w \le C_3  M_p^{1\over w} \left[|\Sigma^{-1}|_{L^1} (M_p J_{n,p,\beta}+r_b)\right]^{1- {r \over w}} \right) \\ \label{eqn:linear_stat_wnorm_rate_bounded_Gclass}
 \ge 1-2p^{-C_b}-2p^{-C_2},
\end{eqnarray}
where the constant $C_3$ depends only on $r,\nu,w$ and $C_1$.
\end{thm}

We remark that the bound (\ref{eqn:linear_stat_wnorm_rate_bounded}) is homogeneous in $\Sigma^{-1}$. If we rescale $\Sigma^{-1}$ by $t > 0$, then the right hand side of (\ref{eqn:linear_stat_wnorm_rate_bounded}) scales by the same factor $t$. Note that Theorem \ref{thm:linear_stat_rate_bounded} is {\it non-asymptotic} and the convergence rates (\ref{eqn:linear_stat_wnorm_rate_bounded}) and (\ref{eqn:linear_stat_wnorm_rate_bounded_Gclass}) hold with probability tending to one polynomially fast in $p$.  Consider the case where $\beta>1$ (short-range dependent case), $\mbf\theta \in \calG_r(\nu, M_p)$ with $r=0$ (true sparsity in $\mbf\theta$), $r_b=0$ ($\vb$ is known). Let $\phi = (\log(p)/n)^{1/2}$ and assume $\log(p)/n \to 0$. For the choice of $\lambda = C M_p \phi$ for some large enough constant $C$, the asymptotic rate of convergence (\ref{eqn:linear_stat_wnorm_rate_bounded_Gclass}) can be simplified as
\begin{equation}
\label{eqn:{eqn:linear_stat_wnorm_rate_bounded_Gclass}_SRD}
|\hat{\mbf\theta} - \mbf\theta|_w = O_\Prob(M_p^{1+1/w} |\Sigma^{-1}|_{L^1} \phi), \qquad w \in [1,\infty].
\end{equation}

The above bound is generally un-improvable. Consider the special case in which $\mathbf x_{i,j}, i,j \in \mathbb Z$, are i.i.d. $N(0, 1)$ . Then $\Sigma = \mathrm{Id}_p$. Let ${\bf e} = (1, 0, \ldots, 0)^\top$ and $\hat{S}_n=n^{-1}\sum_{i=1}^n \vx_i \vx_i^\top$. Let $\lambda = c \phi$, where the constant $c > \sqrt 2$. By elementary calculations, we have with probability going to $1$ that $|(\hat{S}_n - \Sigma){\bf e}|_\infty < \lambda$. Note that (\ref{eqn:{eqn:linear_stat_wnorm_rate_bounded_Gclass}_SRD}) gives $|\hat{\mbf\theta} - \mbf\theta|_w = O_\Prob(\phi)$. Next we shall argue that, 
\begin{equation}
\label{eqn:05111216}
\Prob( |\hat{\mbf\theta} - \mbf\theta|_\infty \ge \phi )  \to 1 \mbox{ as }  p\wedge n \to \infty.
\end{equation}
To this end, let $\tilde {\mbf\theta} = (1-\phi, 0, \ldots, 0)^\top$. Since $\hat \sigma_{1 1} - 1 = O_\Prob(n^{-1/2})$ and $c > \sqrt 2$, we have probability going to one that $|(1-\phi) \hat \sigma_{1 1} - 1| \le \lambda$. Then $\Prob( |\hat S_n \tilde {\mbf\theta} - {\bf e}|_\infty < \lambda) \to 1$. Note that if $|\hat S_n \tilde {\mbf\theta} - {\bf e}|_\infty < \lambda$, then $|\tilde {\mbf\theta}|_1 \ge |\hat {\mbf\theta}|_1$. Hence $|\hat{\mbf\theta} - \mbf\theta|_\infty \ge |1- \hat {\mbf\theta}_1| \ge \phi$. Then (\ref{eqn:05111216}) holds.

\begin{thm}[Exponential-type]
\label{thm:linear_stat_rate_exponential}
Assume (\ref{eqn:coefs_decay}) and (\ref{eq:J290945}). Let $\beta'=\min(2\beta-1, 1/2)$ and
\begin{equation}
\label{eqn:J_exponential}
J_{n,p,\beta,\alpha} = n^{-\beta'} (\log{p})^{2\alpha+2}.
\end{equation}
Let $c_b$ and $C_b $ be constants defined in (\ref{eqn:mean-rate}). Then there exist constants $C_1, C_2, C_3>0$ only depending on $\alpha$, $\beta$, $C_0$ in (\ref{eqn:coefs_decay}), and $C_{\xi,\alpha}$ in (\ref{eq:J290945}), such that for  $\lambda \ge c_b r_b+ C_1 |\mbf\theta|_1 J_{n,p,\beta,\alpha}$ we have that with probability at least $1-2p^{-C_b}-C_2 p^{-C_3}$, $\hat{\mbf\theta}$ satisfies (\ref{eqn:linear_stat_wnorm_rate_bounded}) with $J_{n,p,\beta}$ replaced by $J_{n,p,\beta, \alpha}$.
\end{thm}

\begin{thm}[Polynomial]
\label{thm:linear_stat_rate_polynomial}
Assume (\ref{eqn:coefs_decay}) and (\ref{eq:J291004}) with $q > 4$.  (i) Let $\beta \ge 1-1/q$. Then there exists positive constants $C_1, \ldots, C_4$ such that, for any $\varepsilon > 0$ and $\lambda \ge  \lambda_\varepsilon := c_b r_b + |\mbf\theta|_1 (\varepsilon^{-1} p^{4/q} n^{2/q-1} + C_1 (n^{-1} \log p)^{1/2} )$, with probability at least $p_\varepsilon:=1 - 2p^{-C_b}- C_2(\varepsilon^{q/2}  + p^{-C_3})$, (\ref{eqn:linear_stat_wnorm_rate_bounded}) holds. (ii) Let $1-1/q > \beta > 1/2$. Then the conclusion in (i) holds with $\lambda \ge \lambda_\varepsilon^\diamond := \lambda_\varepsilon +  |\mbf\theta|_1 \varepsilon^{-1/2} p^{2/q} n^{1-2\beta}$.
\end{thm}

Take $\varepsilon = 1$ and thus $\lambda_1=r_b +M_p (p^{4/q} n^{2/q-1} + C_3 (n^{-1} \log p)^{1/2})$, $\lambda_1^\diamond=\lambda_1+M_p p^{2/q}n^{1-2\beta}$. Let  $\tilde \lambda_1=\lambda_1$ if $\beta\ge1-1/q$, and $\tilde \lambda_1=\lambda_1^\diamond$ if  $1-1/q > \beta > 1/2$. For $\theta \in {\cal G}_r(\nu,M_p)$, Theorem \ref{thm:linear_stat_rate_polynomial} implies the rate of convergence
\begin{equation*}
|\hat{\mbf\theta} - \mbf\theta|_w = O_\Prob( M_p^{1\over w} \tilde\lambda_1^{1- {r \over w}} ), \qquad w \in [1,\infty].
\end{equation*}

\begin{table*}
\small
\centering
\caption{Summary: the $\ell^2$ norm rates of convergence (in probability) of $\hat{\mbf\theta}$ under various dependence levels and tail conditions on the linear process $\vx_i = \sum_{m=0}^\infty A_m \vxi_{i-m}$. Dependence index $\beta \in (1,\infty]$, $\beta \in (3/4,1)$ and $\beta \in (1/2,3/4)$ correspond to the SRD (including i.i.d.), weak and strong LRD cases. Sub-Gaussian (including bounded and Gaussian), exponential and polynomial correspond to the moment/tail conditions on $\vxi_i$. We list the rates for $\mbf\theta \in \calG_r(\nu, M_p)$ under the conditions that $r_b=0$ ($\vb$ is observed) and $|\Sigma^{-1}|_{L^1} \le \varepsilon_0^{-1}$: $u_1 = (\log{p} / n)^{1/2}$, $u_2 = (\log{p} / n^{2\beta-1})$, $u_3 = (\log{p})^{2\alpha+2}/n^{1/2}$, $u_4 = (\log{p})^{2\alpha+2}/n^{2\beta-1}$, $u_5 = p^{4/q} / n^{1-2/q}$, and $u_6=p^{2/q}/n^{2\beta-1}$.}
\begin{tabular}{|c|c|c||c|c|}
\hline
 & Sub-Gaussian & Exponential & & Polynomial \\
\hline
$\beta \in (1,\infty]$ & $M_p^{{3-r \over 2}} u_1^{1-{r \over 2}}$ & $M_p^{{3-r \over 2}} u_3^{1-{r \over 2}}$ & $\beta \in (1,\infty)$ & $M_p^{{3-r \over 2}} (u_1 \vee u_5)^{1-{r \over 2}}$ \\
\hline
$\beta \in (3/4,1)$ & $M_p^{{3-r \over 2}} (u_1 \vee u_2)^{1-{r \over 2}}$ &  $M_p^{{3-r \over 2}} u_3^{1-{r \over 2}}$ & $\beta \in [1-1/q,1]$ & $M_p^{{3-r \over 2}} (u_1 \vee u_5)^{1-{r \over 2}}$ \\
\hline
$\beta \in (1/2,3/4)$ & $M_p^{{3-r \over 2}} u_2^{1-{r \over 2}}$ & $M_p^{{3-r \over 2}} u_4^{1-{r \over 2}}$  & $\beta\in(1/2,1-1/q)$ & $M_p^{{3-r \over 2}} (u_1 \vee u_5 \vee u_6)^{1-{r \over 2}}$ \\
\hline
\end{tabular}
\label{tab:convergence-rate_linear_functional}
\end{table*}

The $\ell^2$ norm rates of convergence are summarized in Table \ref{tab:convergence-rate_linear_functional}, which shows several interesting features. First, looking vertically for each column in Table \ref{tab:convergence-rate_linear_functional}, we see that the rates of convergence slow down from SRD to LRD. So the effective sample size shrinks as dependence becomes stronger. Second, horizontal trend of Table \ref{tab:convergence-rate_linear_functional} shows that the rates of convergence becomes worse from sub-Gaussian to exponential-type to polynomial moment conditions.  Third, if the innovations have polynomial moment, then the rate of convergence is determined by a sub-Gaussian term and a polynomial  algebraic tail term.

\begin{rmk}
The boundary cases $\beta=1$ and $3/4$ for Theorem \ref{thm:linear_stat_rate_bounded} can also be dealt with. Assume sub-Gaussian innovations. Following the argument in the latter theorem, the corresponding $\ell^2$ norm rates of convergence in  Table \ref{tab:convergence-rate_linear_functional} is $M_p^{(3-r)/2} u^{1-{r / 2}}$ with $u =  u_1\vee (u_2\log^2 n)$ (resp. $u = (u_1\sqrt{\log n})\vee u_2$) for $\beta=1$ (resp. $\beta=3/4$).
\end{rmk}

\subsection{Sparse Markowitz portfolio allocation}
\label{subsec:MP}

In Markowitz portfolio (MP) allocation \cite{markowitz1952}, the risk of a portfolio of $p$ assets $\vx = (X_1,\cdots,X_p)^\top$ is quantified by the variance of their linear combinations. The optimal portfolio risk for a given amount of expected return $m$ is formulated as
\begin{equation}
\label{eqn:mp_portfolio}
\text{minimize}_{\vw \in \mathbb{R}^p} \quad \Var(\vw^\top \vx) \quad  \text{subject to} \quad \E(\vw^\top \vx) = m.
\end{equation}
If $\vx$ has mean $\mbf\mu$ and covariance matrix $\Sigma$, then the MP is equivalent to (\ref{eqn:constrianed_opt}) and the optimal allocation weights are $\vw^*=m \Sigma^{-1} \mbf\mu / (\mbf\mu^\top \Sigma^{-1} \mbf\mu)$. For a large number of assets, \cite{karoui2010a} showed that the efficient frontier of the MP problem cannot be consistently estimated using the empirical version and the risk is underestimated. Various regularization procedures have been proposed \cite{fanzhangyu2012a,brodieetal2009}. Let $\Delta_p = \mbf\mu^\top \Sigma^{-1} \mbf\mu = \mbf\mu^\top \mbf\theta$, where $\mbf\theta=\Sigma^{-1} \mbf\mu$. Then
$$
\vw^*={m \over \Delta_p} \mbf\theta \quad \text{and} \quad R(\vw^*) = {m^2 \over \Delta_p}.
$$
Note that the MP risk function $R(\vw)=\vw^\top \Sigma \vw$ depends on the distribution of $\vx$ only through the covariance matrix. Let $\hat\vw$ be an estimator of $\vw^*$. We wish to find a $\hat\vw$ such that $R(\hat\vw)$ is close to $R(\vw^*)$.

\begin{defn}
\label{eqn:asymp_optimality}
We say that
 $\hat\vw$ is {\it ratio consistent} if $R(\hat\vw) / R(\vw^*)  \to_\Prob 1$.
\end{defn}

We impose the following assumptions.

\begin{enumerate}[label=\bfseries MP \arabic*:]
\item $|\vw^*|_0 \le s$ and $|\vw^*|_\infty \le C_w$ for some constant $C_w>0$.
\item Let $r_2$ (resp. $r_3$) be the rate of convergence of $\check{S}_n=n^{-1}\sum_{i=1}^n (\vx_i-\mbf\mu)(\vx_i-\mbf\mu)^\top$ and $\bar\vx$:
\begin{eqnarray*}
|\check{S}_n-\Sigma|_\infty = O_\Prob(r_2), \quad |\bar\vx-\mbf\mu|_\infty = O_\Prob(r_3).
\end{eqnarray*}

\item For some constants $K_1, K_2, C > 0$, $|\mu_j| \le K_1, \sigma_{jj} \le K_2$, and $R(\vw^*) \le C$.
\item There exists an estimator $\hat{\mbf\theta}$ satisfying $|\hat{\mbf\theta}-{\mbf\theta}|_1 = O_\Prob(r_1)$.
\end{enumerate}

{\bf MP 1} is a sparsity condition on the oracle portfolio weights. {\bf MP 2} is a high-level assumption on the concentration of maximum norms on sample mean and covariances about their expectations, which can be fulfilled for a broad range of moment and dependence conditions on $\vx_i$. {\bf MP 3} is a regularity condition excluding assets with extremely large mean returns and unbounded risks. {\bf MP 4} requires the existence of an estimator for the linear functional $\mbf\theta$, which can be verified by our main result in Section \ref{sec:rates_linear_statistics} under mild conditions. As a natural condition to get consistency, we assume $\max(r_1,r_2,r_3)=o(1)$ as $n,p\to \infty$.

The intuition of the proposed estimator for $\vw^*$ is explained as follows. Since $\vw^*$ is sparse, so is $\mbf\theta$ and therefore we can seek a sparse estimator $\hat{\mbf\theta}$ such that $|\hat{\mbf\theta}-\mbf\theta|_1 \to_\Prob 0$. Then, we expect
$$
|\mbf\mu^\top \mbf\theta - \bar\vx^\top \hat{\mbf\theta} | \le |\mbf\mu|_\infty |\hat{\mbf\theta}-\mbf\theta|_1 + |\bar\vx-\mbf\mu|_\infty |\hat{\mbf\theta}|_1 \to_\Prob 0
$$
so that $|\hat\vw-\vw^*|$ is small and $R(\hat\vw)$ is close to $R(\vw^*)$. Now, we describe our method, which contains two steps. First, we estimate $\mbf\theta$ by
\begin{equation}
\label{eqn:mp_step1}
\text{minimize}_{\mbf\eta \in \mathbb{R}^p} |\mbf\eta|_1 \qquad \text{subject to} \quad |\hat{S}_n \mbf\eta - \bar\vx|_\infty \le \lambda.
\end{equation}
Denote the solution by $\hat{\mbf\theta}$. Then, we compute $\hat\Delta_{p,n} = \bar\vx^\top \hat{\mbf\theta}$ and $\hat\vw = m \hat{\mbf\theta} / \hat\Delta_{p,n}$.

\begin{prop}
\label{prop:mp_property}
Fix the mean return level $m$ and assume {\bf MP 1--4}. In (\ref{eqn:mp_step1}) choose $\lambda \ge C(\Delta_p s (r_2 + r_3^2) + r_3)$, where $C> 0$ is a sufficiently large constant. If $s r_1 + \Delta_p s^2 (r_2 + r_3^2) = o(1)$, then $\hat\vw$ is ratio consistent.
\end{prop}

\begin{rmk}
In Proposition \ref{prop:mp_property}, $s r_1 + \Delta_p s^2(r_2 + r_3^2) = o(1)$ is a natural condition since $r_1$ and $r_2$ control the error in estimating $\mbf\theta$ and $\Sigma$, while $s$ and $\Delta_p$ reflect the difficulty of the high-dimensional problem. In particular, $\Delta_p$ cannot diverge too fast in order to get ratio consistency in the risk: if $\Delta_p$ diverges faster, then $R(\vw^*)\to0$ so quickly that makes any estimation procedure inferior to the accurate oracle. Therefore, characterization of the optimality of our procedure depends on the moment and the dependence conditions on $\vx_i$ through the rates $r_1, r_2$, and $r_3$. For example, applying Proposition \ref{prop:mp_property} to SRD time series ($\beta>1$) with sub-Gaussian tails, we may take $r_2 = r_3 = \sqrt{\log{p}/n}$ and $r_1 = |\Sigma^{-1}|_{L^1} s^2 \sqrt{\log{p}/n}$. Then, a sufficient condition for ratio consistency is $(s |\Sigma^{-1}|_{L^1} + \Delta_p) s^2 \sqrt{\log{p}/n} = o(1)$.
\end{rmk}

\subsection{Sparse full-sample optimal linear prediction}
\label{subsec:sfso_linpred}

In this section we consider the optimal linear prediction for a univariate time series. Let $\xi_i$ be i.i.d. mean-zero random variables with unit variance and
\begin{equation}\label{eq:one_dim_ts}
X_i=\sum_{m=0}^{\infty}a_m\xi_{i-m}
\end{equation}
be a mean-zero univariate linear process, where $|a_m| \le C_0 (1\vee m)^{-\beta}$ for $m \ge 0$ and $\beta>1/2$. Denote $\vx= (X_1,\cdots,X_n)^\top$ and $\Gamma=\E(\vx \vx^\top)$ as the auto-covariance matrix of $\vx$. If $\vx$ is viewed as an $n$-dimensional observation, then $\Gamma$ is the covariance matrix of $\vx$. The $\ell^2$ optimal one-step linear predictor for $X_{n+1}$ based on the past sample is $X_{n+1}^* = \sum_{i=1}^n \theta_i X_{n+1-i}$, where the coefficient vector $\mbf\theta=(\theta_1,\cdots,\theta_n)^\top$ is determined by the Yule-Walker equation
\begin{equation}
\label{eqn:opt_linpred_oracle}
\mbf\theta = \Gamma^{-1} \mbf\gamma
\end{equation}
and $\mbf\gamma$ is the shifted first row of $\Gamma$. Let $\breve \gamma_s=n^{-1} \sum_{t=1}^{n-|s|}X_tX_{t+|s|}$ be the sample auto-covariances and \[
\kappa(x)=\begin{cases}
1,&\mbox{ if }|x|\le 1,\\
g(|x|),&\mbox{ if }1<|x|\le c,\\
0,&\mbox{ if }|x|> c,
\end{cases}
\]
where the function $g(\cdot)$ satisfying $|g(x)|<1$, and $c > 1$ is a constant. \cite{mcmurry2010banded} proposed the flat-top tapered auto-covariance matrix estimator
 \[
 \hat\Gamma_n=(\hat\gamma_{|j-k|})_{1\le j,k\le n}, \quad \text{where } \hat\gamma_s=\kappa(|s|/l)\breve \gamma_s, \quad |s|\le n.
 \]

It has been shown in \cite{mcmurrypolitis2015} that optimal linear prediction based on full time series sample can be achieved by
\begin{equation}
\label{eqn:mp_linpred}
\tilde{\mbf\theta} = \hat\Gamma_n^{-1} \hat{\mbf\gamma}_n.
\end{equation}
If the best linear predictor can be approximated by a sparse linear combination in the full sample, \cite{chen2015} proposed a sparse full-sample optimal (SFSO) linear predictor $\hat{\mbf\theta}$ that solves
\begin{equation}
\label{eqn:sfso_linpred}
\text{minimize}_{\mbf\eta \in \mathbb{R}^p} |\mbf\eta|_1 \qquad \text{subject to} \quad |\hat\Gamma_n \mbf\eta - \hat{\mbf\gamma}_n|_\infty \le \lambda,
\end{equation}
which has a better convergence rate than $\tilde{\mbf\theta}$ in (\ref{eqn:mp_linpred}). Let $\gamma_0 = \E X_1^2$. The $\ell^2$ risk function $R(\vw) = \E (\vw^\top \vx - X_{n+1})^2 = \vw^\top \Gamma \vw - 2\vw^\top \mbf\gamma + \gamma_0$ is a natural criterion to assess the quality of estimators. Note that the oracle risk for (\ref{eqn:opt_linpred_oracle}) is $R(\mbf\theta) = \gamma_0 - \mbf\gamma^\top \Gamma^{-1} \mbf\gamma = \gamma_0 - \mbf\theta^\top \Gamma \mbf\theta$. It was established in \cite{chen2015} that the SFSO is consistent for estimating the best sparse linear predictor in the $\ell^2$-norm. Here, we use the ratio consistency criterion to assess the SFSO compared with the oracle predictor (\ref{eqn:opt_linpred_oracle}). We shall make the following assumptions.

\begin{enumerate}[label=\bfseries OLP \arabic*:]
\item $|\mbf\theta|_0 \le s$ and $|\mbf\theta|_\infty \le C_0$.
\item For some constants $K_1, C > 0$, $|\Gamma|_\infty \le K_1$ and $R(\mbf\theta) \ge C$.
\end{enumerate}

Assumptions {\bf OLP 1-2} are parallel to {\bf MP1, 3}. The oracle risk $R(\mbf\theta)$ is lower bounded to rule out the unpractical cases where the prediction can be perfectly done using past observations.

\begin{prop}
\label{prop:sfso_linpred_property}
Let $(X_i)$ be a linear process defined in (\ref{eq:one_dim_ts}) and $\|\xi_i\|_q<\infty$ for some $q \ge 4$. Let $r_4 = r_0 + r_5$, where $r_0 = l^{-\beta}$ or $l^{1-2\beta}$ if $\beta>1$ or $1>\beta>1/2$ and $r_5 = (\log{l}) n^{-\beta'} \|\xi_0\|_q^2$, where we recall $\beta'=\min(2\beta-1, 1/2)$. Let $\lambda\ge C(|\mbf\theta|_1+1)r_4$ in (\ref{eqn:sfso_linpred}). Then we have 
\begin{equation}
\label{eqn:sfso_rate}
|\hat{\mbf\theta}-{\mbf\theta}|_1 = O_\Prob(D(5\lambda|\Gamma^{-1}|_{L^1})).
\end{equation}
Assume further {\bf OLP 1-2}. If $D(5\lambda|\Gamma^{-1}|_{L^1}) = o(1)$, then the SFSO linear predictor is ratio consistent.
\end{prop}

\begin{rmk}
In \cite{mcmurrypolitis2015}, the $\ell^2$ rate of convergence  $|\tilde{\mbf\theta}-\mbf\theta|_2=O_\Prob(ln^{-1/2} + \sum_{i=l}^\infty |\gamma_l|)$, where $l$ is the bandwidth of the flap-top matrix taper. Therefore, $\tilde{\mbf\theta}$ is not consistent in the long-range dependence setting. Finite sample performances based on the relative risk are assessed in Section \ref{subsec:optimal_linear_prediction}. On the other hand, the rate obtained in (\ref{eqn:sfso_rate}) is sharper than \cite[Theorem 2]{chen2015} if $\xi_i$ has a polynomial tail. This is due to the tighter concentration inequality for $|\hat\Gamma_n-\Gamma|_\infty$ with the auto-covariance structures (Lemma \ref{lem:bound_max_autocovariances}).
\end{rmk}

\section{Simulation studies}
\label{sec:simulation}

Here we shall study how the dependence, dimension and the innovation moment condition affect the finite sample performance of the linear functional estimate (\ref{eqn:functional}). We simulate a variety of time series of length $n = 100, 200$ while fixing the dimension $p = 100$. We consider three dependence levels: $\beta = 2$, $0.8$, $0.6$, corresponding to the SRD ($\beta > 1$), the weak LRD ($1 > \beta > 3/4$) and the strong LRD ($3/4 > \beta > 1/2$) processes. The coefficient matrices $A_m$ are formed by i.i.d. Gaussian random entries $N(0, p^{-1})$ multiplied by the decay rates $m^{-2}, m^{-0.8}$ and $m^{-0.6}$, respectively. Then $80\%$ randomly selected entries of $A_m$ are further set to zero. Four types of i.i.d. innovations are included: uniform $[-3^{1/2}, 3^{1/2}]$, standard normal, standardized double-exponential and Student-$t_3$.

A data splitting procedure is used to select the optimal tuning parameters. To preserve the temporal dependence, we split the data into two halves: the first half is used for estimation and the second half is used for testing. In the linear functional $\mbf\theta=\Sigma^{-1}\mathbf{b}$,  $\vb$ is chosen such that the coefficient vector $\mbf\theta$ has $80\%$ zeros and $20\%$ i.i.d. non-zeros. Each simulation setup is repeated for 100 times and we report the averaged performance for the ``block data-splitting" and the ``oracle" estimate. Here, the block data-splitting estimate refers to the validation procedure on the second half testing data from the data splitting procedure and the oracle estimate refers to the validation procedure using the true covariance matrix. Validation procedures are used to select the tuning parameter $\lambda$ that minimize the  $\ell^2$ loss $|\hat\Sigma_{\text{test}}\hat{\mbf\theta}_{\text{train}}(\lambda)-\vb|$ and $|\Sigma\hat{\mbf\theta}_{\text{train}}(\lambda)-\vb|$ for the data-adaptive estimate and the oracle estimate respectively. Results are shown in Tables  \ref{tab:lin_n=100}, \ref{tab:lin_n=200}, and Figures  \ref{fig:lin_error-curve_n=100}, \ref{fig:lin_error-curve_n=200}.

A number of conclusions can be drawn from the simulation results. First, we look at the selected tuning parameters by the block data-splitting procedure. Tables \ref{tab:lin_n=100} and \ref{tab:lin_n=200} suggest that the optimal tuning parameters are data-adaptive (w.r.t. the dependence level, tail condition and sample size) in the sense that they are getting closer to the optimal constraint parameters validated by the oracle as the sample size increases. In particular, for each setup $(n,p)$, the optimal constraint parameter becomes larger, as (i) the dependence gets stronger, (ii) the tail gets thicker, and (iii) the sample size decreases. This is consistent with our theoretical analysis in Section \ref{sec:methods}; see Theorem \ref{thm:linear_stat_rate_bounded}--\ref{thm:linear_stat_rate_polynomial}.

\begin{table}[htp]
\caption{The optimal constraint parameter $\lambda$ selected by the oracle and the block data-splitting procedure in the Dantzig selector type estimate for $\Sigma^{-1} \vb$. Standard deviations are shown in the parentheses. $p = 100$ and $n = 100$.}
\label{tab:lin_n=100}
\begin{center}
\begin{tabular}{cccccc}
\hline
&  & bounded & Gaussian & double-exp & Student-$t$\\
\hline
\multirow{4}{*}{$\beta=2$} &\multirow{2}{*}{oracle} & 0.1221 & 0.1289 & 0.1225 & 0.1340 \\
& & (0.0236) & (0.0244) & (0.0241) & (0.0245) \\
&\multirow{2}{*}{block} & 0.1939 & 0.1961 & 0.1842 & 0.2291 \\
& & (0.0533) & (0.0540) & (0.0490) & (0.0808) \\
\hline
\multirow{4}{*}{$\beta=0.8$} &\multirow{2}{*}{oracle} & 0.2419 & 0.2470 & 0.2434 & 0.2549 \\
& & (0.0424) & (0.0446) & (0.0469) & (0.0475) \\
&\multirow{2}{*}{block} & 0.4227 & 0.4655 & 0.4188 & 0.4806 \\
& & (0.1216) & (0.1424) & (0.1267) & (0.1543) \\
\hline
\multirow{4}{*}{$\beta=0.6$} &\multirow{2}{*}{oracle} & 0.4835 & 0.4817 & 0.4855 & 0.4875 \\
& & (0.0798) & (0.0868) & (0.0840) & (0.0784) \\
&\multirow{2}{*}{block} & 0.9147 & 0.9789 & 0.9327 & 0.9936 \\
& & (0.2640) & (0.2897) & (0.2906) & (0.2930) \\
\hline
\end{tabular}
\end{center}
\end{table}

\begin{table}[htp]
\caption{The optimal constraint parameter $\lambda$ selected by the oracle and the block data-splitting procedure in the Dantzig selector type estimate for $\Sigma^{-1} \vb$. Standard deviations are shown in the parentheses. $p = 100$ and $n = 200$.}
\label{tab:lin_n=200}
\begin{center}
\begin{tabular}{cccccc}
\hline
&  & bounded & Gaussian & double-exp & Student-$t$\\
\hline
\multirow{4}{*}{$\beta=2$} &\multirow{2}{*}{oracle} & 0.0763 & 0.0758 & 0.0797 & 0.0875 \\
& & (0.0150) & (0.0138) & (0.0156) & (0.0170) \\
&\multirow{2}{*}{block} & 0.1062 & 0.1032 & 0.1109 & 0.1261 \\
& & (0.0211) & (0.0236) & (0.0260) & (0.0386) \\
\hline
\multirow{4}{*}{$\beta=0.8$} &\multirow{2}{*}{oracle} & 0.1555 & 0.1544 & 0.1555 & 0.1627 \\
& & (0.0266) & (0.0253) & (0.0275) & (0.0292) \\
&\multirow{2}{*}{block} & 0.2485  & 0.2473 & 0.2554 & 0.2594 \\
& & (0.0573) & (0.0515) & (0.0590) & (0.0624) \\
\hline
\multirow{4}{*}{$\beta=0.6$} &\multirow{2}{*}{oracle} & 0.3364 & 0.3307 & 0.3349 & 0.3353 \\
& & (0.0527) & (0.0518) & (0.0540) & (0.0466) \\
&\multirow{2}{*}{block} & 0.5673 & 0.5472 & 0.5743 & 0.5544 \\
& & (0.1193) & (0.1159) & (0.1207) & (0.1245) \\
\hline
\end{tabular}
\end{center}
\end{table}

Second, from Figure \ref{fig:lin_error-curve_n=100} and Figure \ref{fig:lin_error-curve_n=200}, it is clear that the Student-$t$(3) innovations, which have the infinite forth moment, uniformly perform worse than the innovations with bounded support, Gaussian tail and exponential tail. This empirically justifies our theoretical results regarding the moment/tail condition; see the asymptotic rates of convergence in Section \ref{sec:methods}. Moreover, similarly as the optimal tuning parameter, the estimation error also increases, as (i) the dependence gets stronger and (ii) the sample size decreases. In addition, the effect of the innovation distribution becomes relatively smaller when dependence strength increases.

\begin{figure}[htbp] 
   \centering
	\subfigure[$\beta = 2$.] {\label{subfig:lin_error-curve_n=100_beta=2}\includegraphics[width=1.9in]{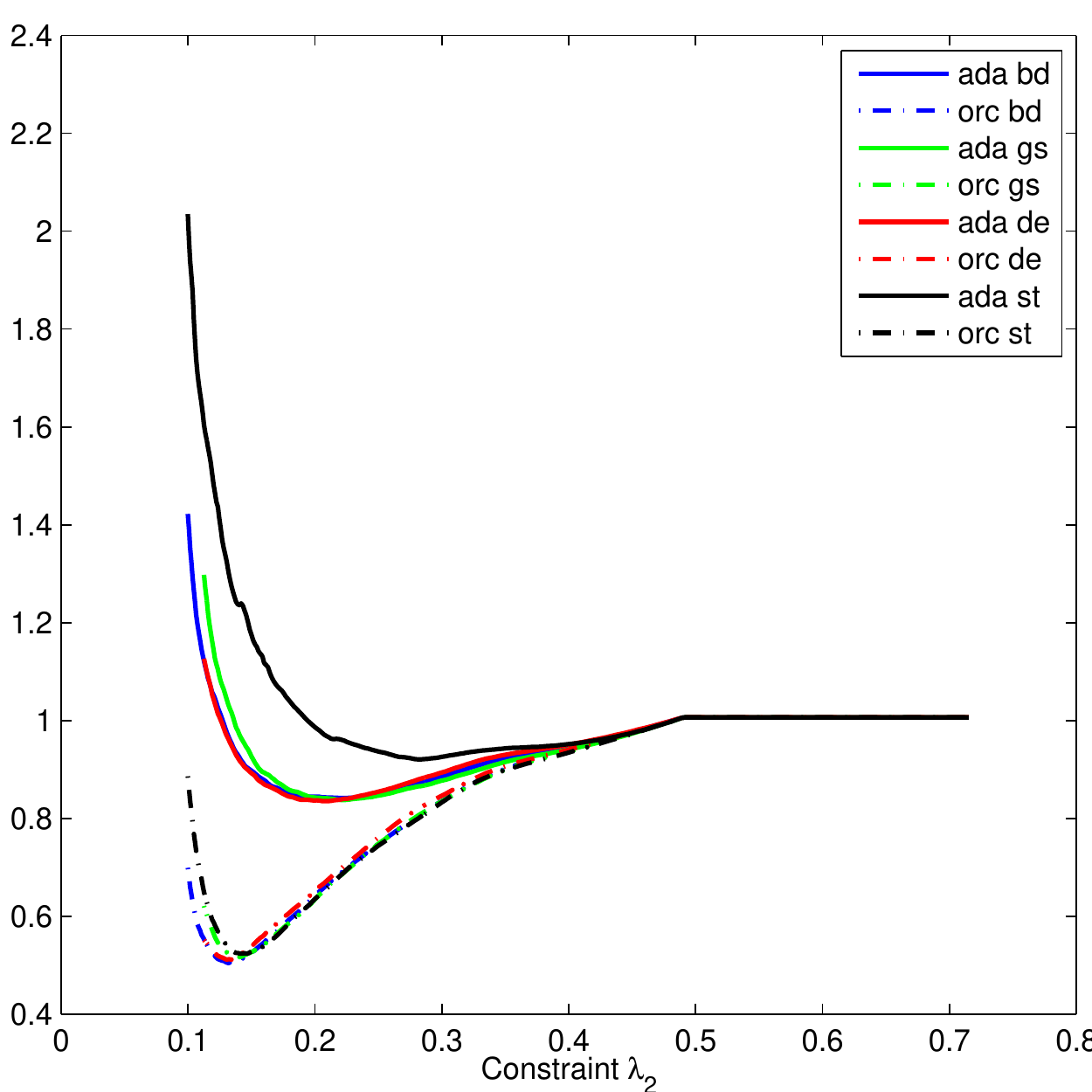}}
	\subfigure[$\beta = 0.8$.] {\label{subfig:lin_error-curve_n=100_beta=0.8}\includegraphics[width=1.9in]{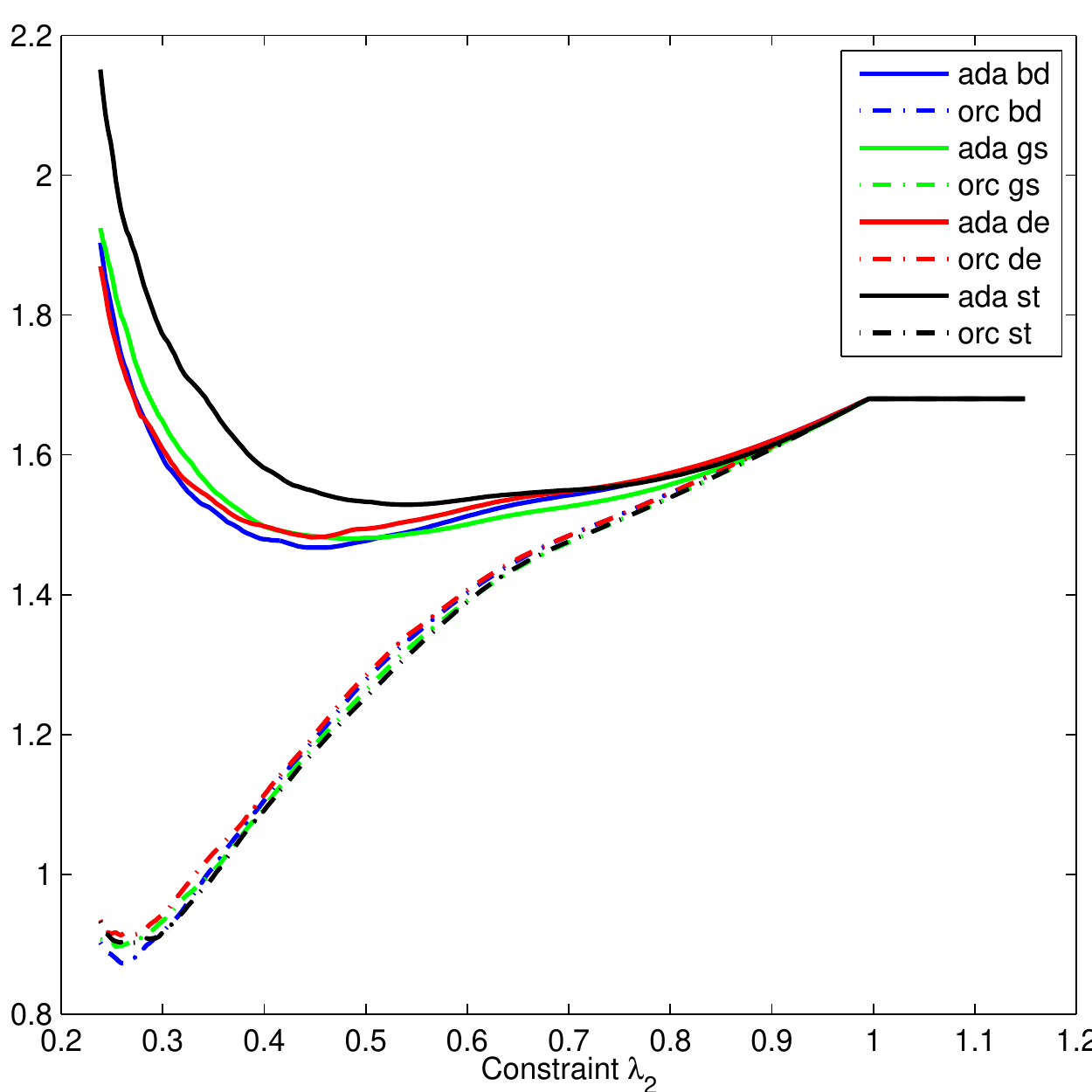}}
	\subfigure[$\beta = 0.6$.] {\label{subfig:lin_error-curve_n=100_beta=0.6}\includegraphics[width=1.9in]{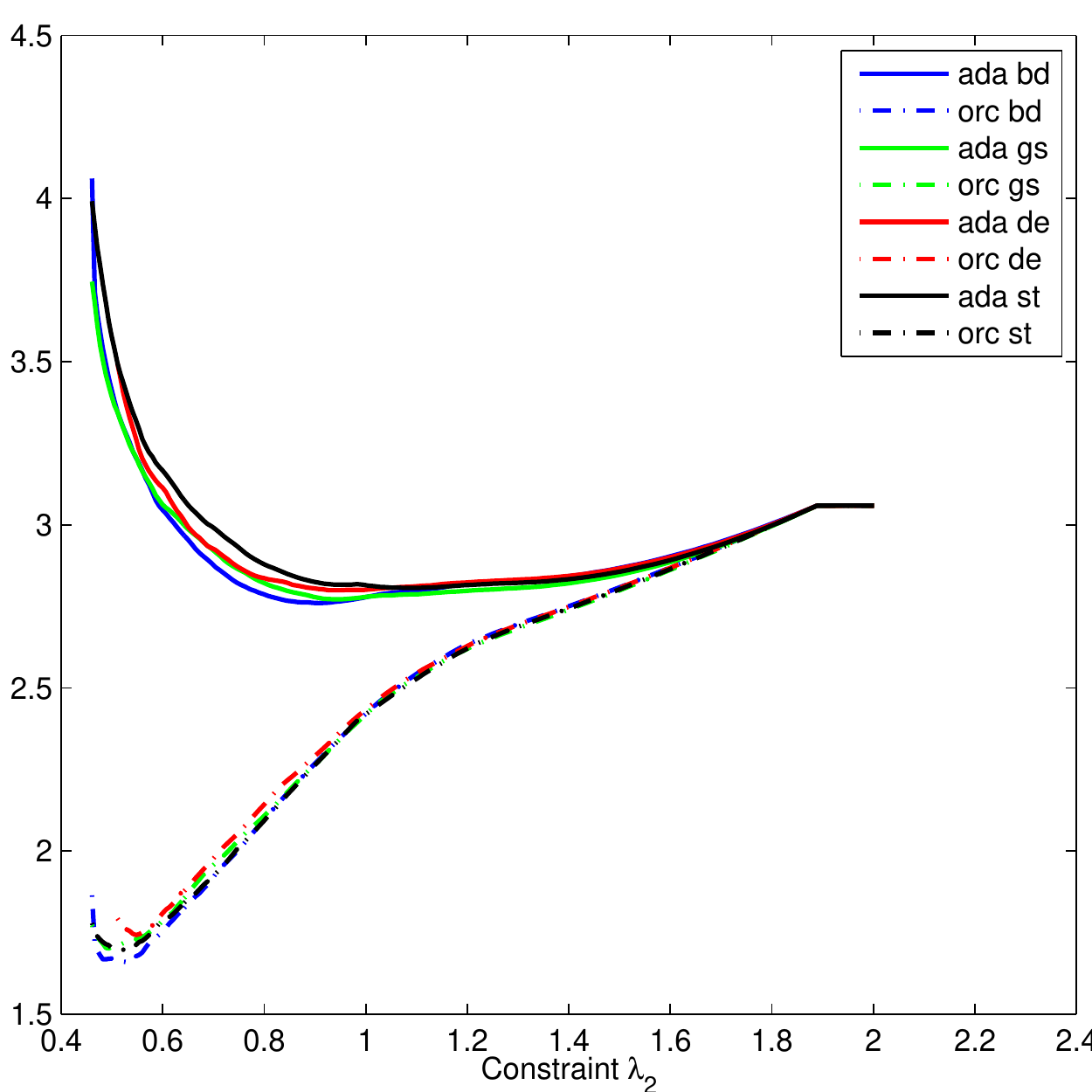}}
   \caption{Error curves under the $\ell^2$ loss for the linear statistics estimate for $p = 100$ and $n = 100$. $x$-axis is the threshold, $y$-axis is the quadratic error. `ada' means adaptive block data-splitting procedure and `orc' means the oracle procedure. `bd', `gs', `de' and `st' denote bounded, Gaussian, double-exponential and Student-$t$ distributions, respectively.}
   \label{fig:lin_error-curve_n=100}
\end{figure}

\begin{figure}[htbp] 
   \centering
	\subfigure[$\beta = 2$.] {\label{subfig:lin_error-curve_n=200_beta=2}\includegraphics[width=1.9in]{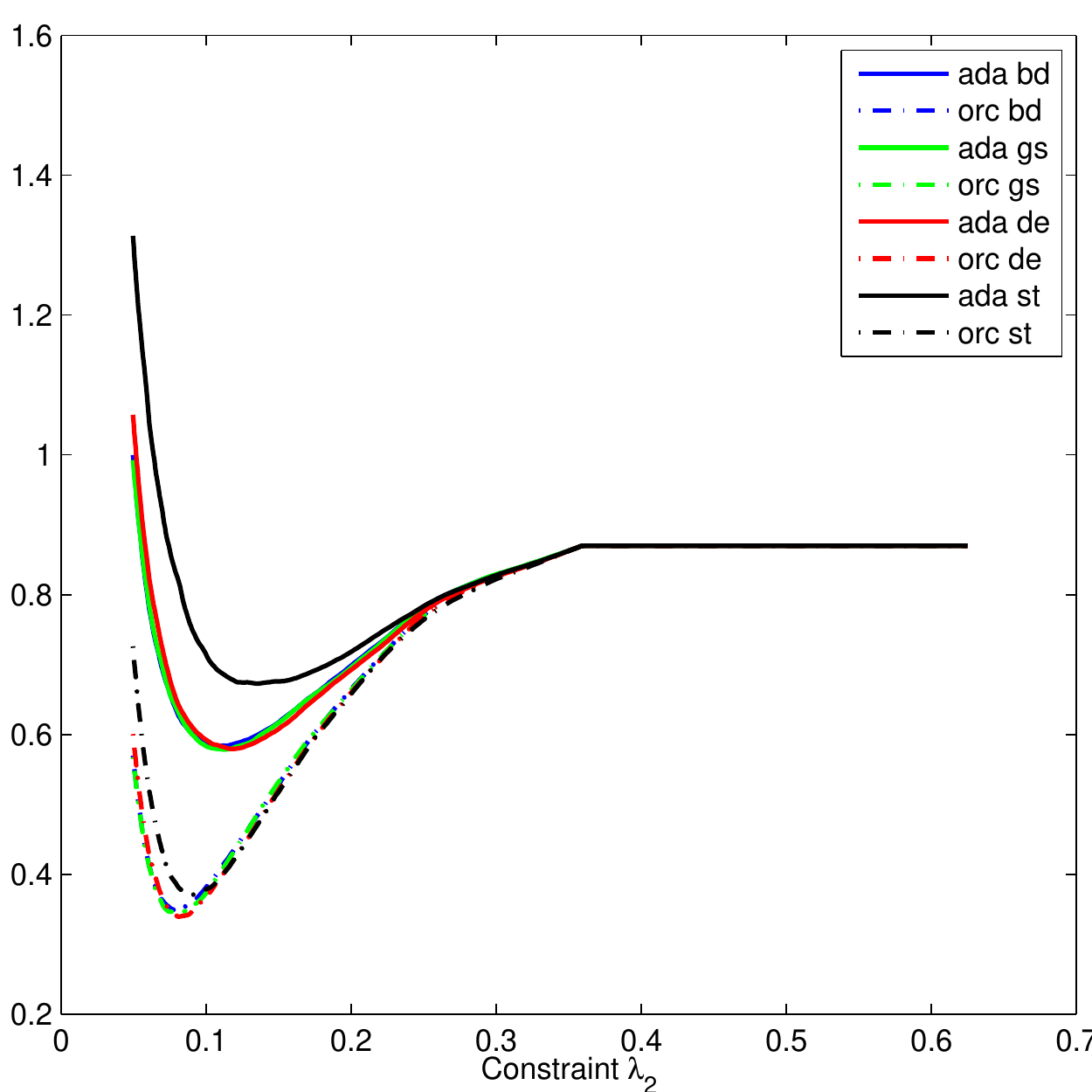}}
	\subfigure[$\beta = 0.8$.] {\label{subfig:lin_error-curve_n=200_beta=0.8}\includegraphics[width=1.9in]{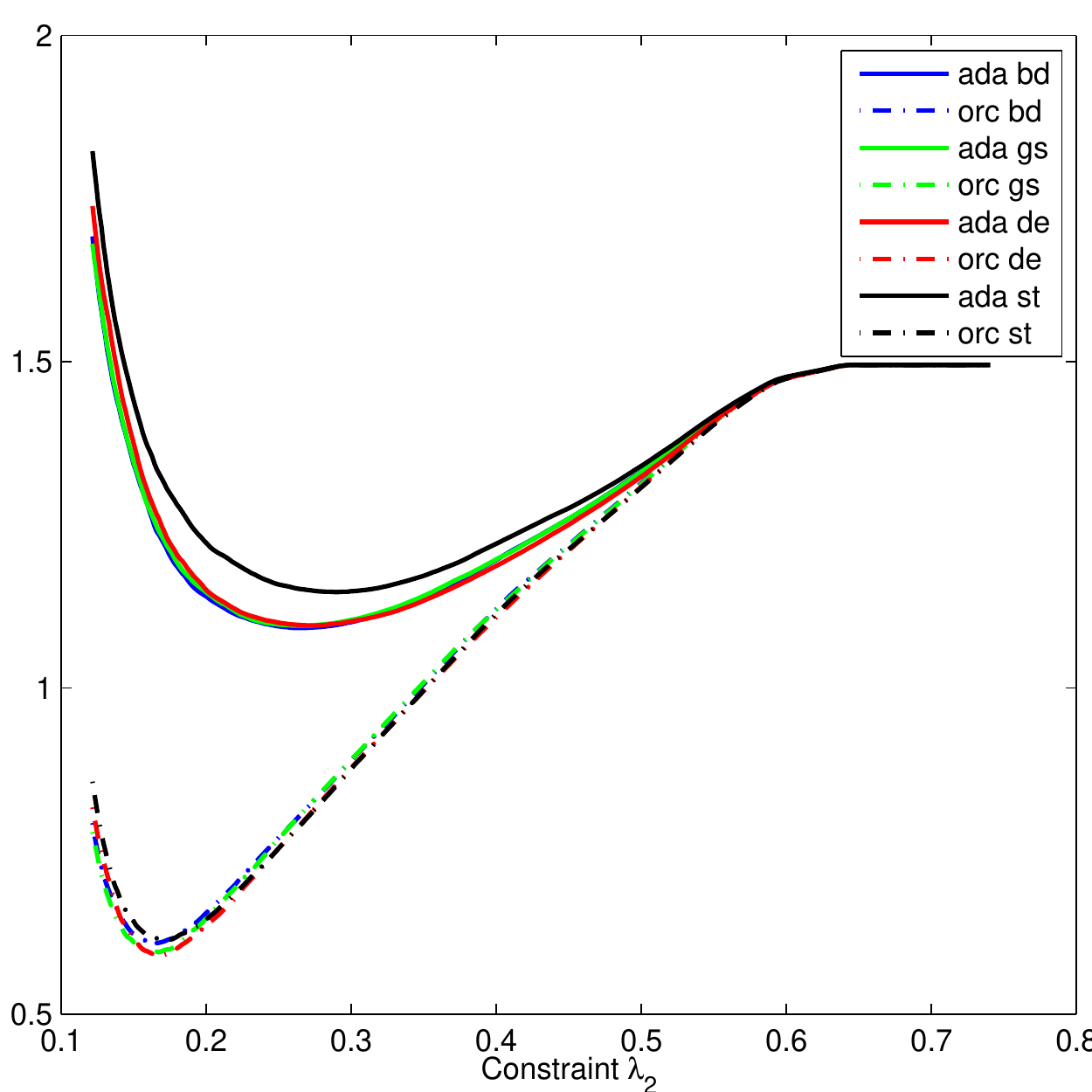}}
	\subfigure[$\beta = 0.6$.] {\label{subfig:lin_error-curve_n=200_beta=0.6}\includegraphics[width=1.9in]{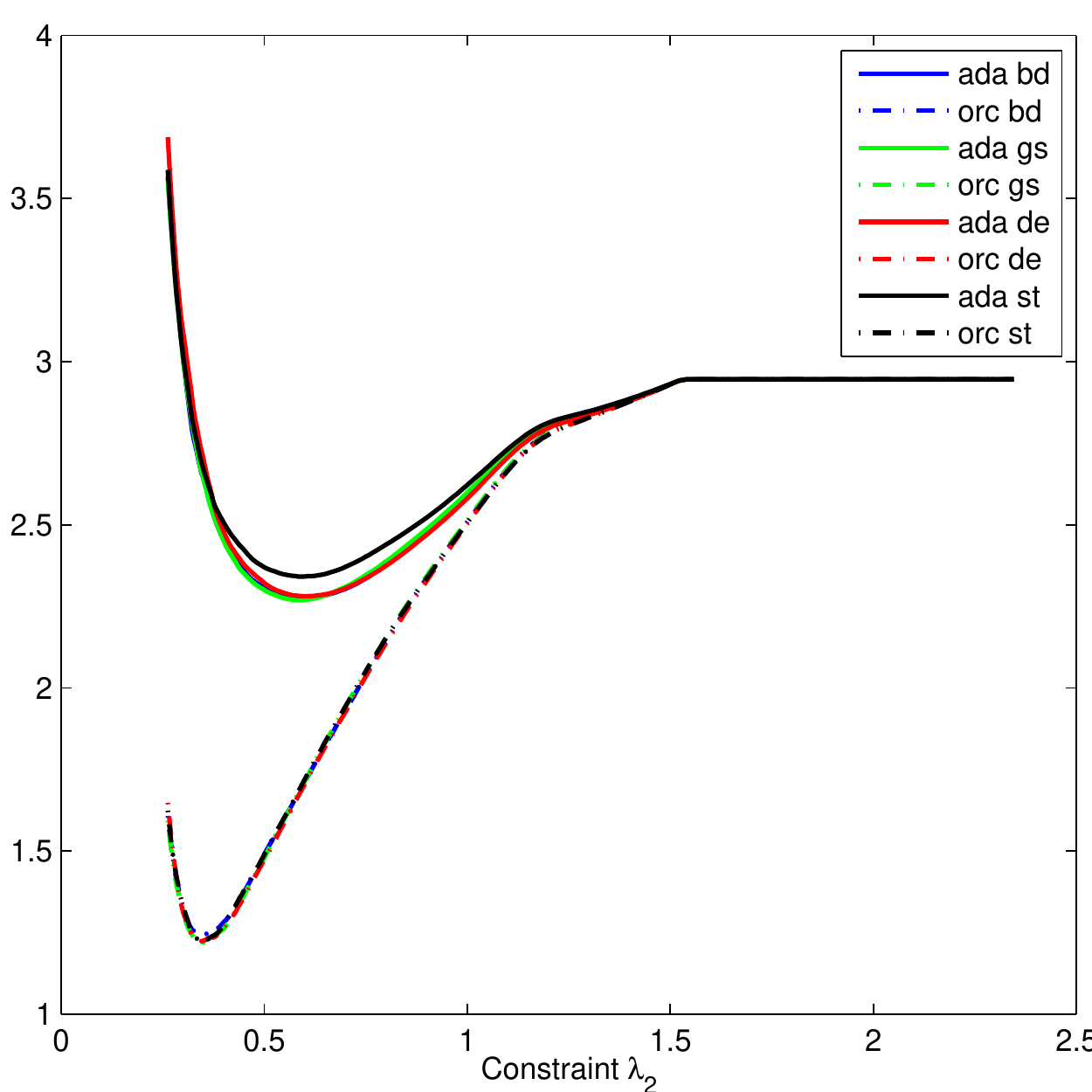}}
   \caption{Error curves under the $\ell^2$ loss for the linear statistics estimate for $p = 100$ and $n = 200$. $x$-axis is the threshold, $y$-axis is the quadratic error. `ada' means adaptive block data-splitting procedure and `orc' means the oracle procedure. `bd', `gs', `de' and `st' denote bounded, Gaussian, double-exponential and Student-$t$ distributions, respectively.}
   \label{fig:lin_error-curve_n=200}
\end{figure}

\subsection{Optimal linear prediction}
\label{subsec:optimal_linear_prediction}

We verify the ratio consistency of the sparse full sample optimal linear predictor in Section \ref{subsec:sfso_linpred} on finite samples. Partially following the setup in \cite{chen2015}, we simulate stationary Gaussian time series from two models
\begin{enumerate}
\item AR(14) model: $X_i = \sum_{j=1}^{14} \theta_j X_{i-j} + e_i$, where $\theta_1=-0.3, \theta_3=0.7, \theta_{14}=-0.2$, and the rest of $\theta_j=0$. The errors $e_i$ are i.i.d. $N(0,1)$.
\item AR(1) model: $X_i = \theta X_{i-1} + e_i$, where $\theta = -0.5$ and $e_i$ are i.i.d. $N(0,1)$.
\end{enumerate}
We take the following competitors of the SFSO:  the two versions of ridge corrected shrinkage predictors (FSO-Ridge, FSO-Ridge-Thr) in \cite{chen2015} and the thresholding (FSO-Th-Raw, FSO-Th-Thr), shrinkage to a positive definite matrix (FSO-PD-Raw, FSO-PD-Thr) and white noise (FSO-WN-Raw, FSO-WN-Thr) predictors in \cite{mcmurrypolitis2015}. We also run the R function \texttt{ar()} as the benchmark with the default parameter that uses the Yule-Walker solution with order selection by the AIC. We fix the tuning parameter $\lambda = \sqrt{\log(n)/n}$ for the SFSO. We try two sample sizes $n=200,500$. We follow the empirical rule for choosing the bandwidth parameter $l$ for all competitors in \cite{mcmurrypolitis2015}. The performance of those estimators are assessed by the estimated relative risks. All numbers in Table \ref{tab:ar(14)} and \ref{tab:ar(1)} are reported by averaging 1000 simulation times. In both AR(1) and AR(14) models, our simulation shows that the SFSO is very close to the oracle risk. This confirms our theoretical findings in Proposition \ref{prop:sfso_linpred_property}. On the other hand, the relative risk for shrinkage based predictors tend to perform worse relatively to the oracle. It also is observed that the AR and SFSO predictors are comparably the best among all predictors considered here. If we look at the estimation errors, there is a sizable improvement for the SFSO over the AR due to sparsity; c.f. \cite{chen2015}. The improved performance for SFSO on the AR(14) model is larger than other methods (except AR) on the AR(1) model, which is explained by the sparsity structure in the oracle linear predictor.

\begin{table}[ht!]
\caption{Estimated relative risks for the AR(14) models for $n=200$ and $n=500$. The oracle risk is one. Standard errors are shown in parentheses. All method symbols are consistent with \cite{chen2015}.}
\label{tab:ar(14)}
	\centering
	\begin{tabular}{r c c}
	\hline
	 & $n=200$ & $n=500$ \\
	 \hline
	 AR          &  1.1168 (0.0535)   & 1.0336 (0.0159) \\
	 SFSO  & 1.1173 (0.0851) &  1.0455 (0.0256) \\
	 FSO-Ridge  & 1.3443 (0.2433) & 1.2897 (0.4119) \\
	 FSO-Ridge-Thr  & 1.4076 (0.3525) & 1.3913 (0.8883) \\
	 FSO-Th-Raw  & 2.4623 (3.3663) & 13.4350 (74.0697) \\
	 FSO-Th-Shr  & 1.6530 (0.8478) & 3.3540 (9.6394) \\
	 FSO-PD-Raw  & 1.4930 (0.3388) & 1.4475 (0.5842) \\
	 FSO-PD-Shr  & 1.4584 (0.3127) & 1.3361 (0.2087) \\
	 FSO-WN-Raw  & 2.1798 (2.9911) & 10.7390 (62.8709) \\
	 FSO-WN-Shr  & 1.6859 (1.2386) & 4.1574 (15.2984) \\
	 \hline
	\end{tabular}
\end{table}

\begin{table}[ht!]
\caption{Estimated relative risks for the AR(1) models for $n=200$ and $n=500$. Standard errors are shown in parentheses. Standard errors are shown in parentheses. The oracle risk is one. All method symbols are consistent with \cite{chen2015}.}
	\label{tab:ar(1)}	
	\centering
	\begin{tabular}{r c c}
	\hline
	 & $n=200$ & $n=500$ \\
	 \hline
	 AR       & 1.0171 (0.0270) & 1.0062 (0.0108) \\
	 SFSO  & 1.0310 (0.0274) & 1.0120 (0.0104) \\
	 FSO-Ridge  & 1.0314 (0.0188) & 1.0128 (0.0103) \\
	 FSO-Ridge-Thr  & 1.0530 (0.0383) & 1.0155 (0.0182) \\
	 FSO-Th-Raw  & 1.1055 (0.1520) & 1.0161 (0.0232) \\
	 FSO-Th-Shr  & 1.0984 (0.1294) & 1.0161 (0.0232) \\
	 FSO-PD-Raw  & 1.0367 (0.0224) & 1.0138 (0.0109) \\
	 FSO-PD-Shr  & 1.0310 (0.0187) & 1.0122 (0.0088) \\
	 FSO-WN-Raw  & 1.0694 (0.0608) & 1.0161 (0.0232) \\
	 FSO-WN-Shr  & 1.0645 (0.0519) & 1.0161 (0.0232) \\
	 \hline
	\end{tabular}
\end{table}

\section{Real data analysis}
\label{sec:real_data_analysis}

\subsection{Task classification for fMRI data}
\label{sec:fMRI_data}

In this section, we apply the methods in Section \ref{sec:methods} to a real data for the cognitive states classification
using the fMRI data. This publicly available dataset is called \emph{StarPlus}.
In this fMRI study, during the first four seconds, a subject sees a picture such as ${+ \over *}$, i.e. the symbol stimulus. Then after another four seconds for a blank screen, the subject is presented a sentence like ``The plus sign is above on the star sign.", i.e. the semantic stimulus, which also lasts for four seconds, followed by an additional four blank seconds. One Picture/Sentence switch is called a trial and 20 such trials are repeated in the study. In each trial, the first eight seconds are considered as the ``Picture" (abbr. ``P") state and the last eight seconds belong to the ``Sentence" (abbr. ``S") state. Sampling rate of the fMRI image slides is 2Hz and each slide is a 2-D image containing seven anatomically defined Regions of Interests (ROIs).\footnote{The seven ROIs are: \texttt{`CALC', `LDLPFC', `LIPL', `LIPS', `LOPER', `LT', `LTRIA'}.} In this data analysis, we use four ROIs\footnote{The selected four ROIs used in our analysis are: \texttt{CALC, LIPL, LIPS, LOPER}.} and each ROI may have a varying number of voxels (i.e. the 3-D pixels) for different subjects. The four ROIs contain 728--1120 voxels in total, depending on the subject. Therefore, for each subject, we have two multi-channel time-course data matrices: one has $320$ time points with ``S" state and the other has $320$ time points with ``P" state, both having the dimension $p$ equal to the number of voxels in that subject. Therefore, this is a high-dimensional time series dataset ($p > n$). We assume that the overall time-course data are covariance stationary and standardize the data to unit diagonal entries in the covariance matrix. The goal of this study is to classify the state of subject (``P" and ``S") based on the past fMRI signals.

\begin{table*}
\centering
\small
\caption{Accuracy of the RLDA classifier (\ref{eqn:RLDA_Bayes_rule}), with different estimates of the pooled covariance matrix $\Sigma$ (with thresholding), its inverse $\Sigma^{-1}$ (graphical Lasso), its linear functional $\Sigma^{-1} (\hat{\mbf\mu}_{\text{P}}-\hat{\mbf\mu}_{\text{S}})$ (\ref{eqn:functional}), and the GNB classifier. Four ROIs -- \texttt{CALC, LIPL, LIPS, LOPER} -- are used in the ``Picture/Sentence" dataset.}
\begin{tabular}{|c|c|c|c|c|c|}
\hline
Subject & \# Voxels & Thresholded $\Sigma$ & Graphical Lasso $\Sigma^{-1}$ & Linear functional & GNB \\
\hline
04799 & 846 & $85\%$ & $90\%$ & $95\%$ & $80\%$ \\
04820 & 728 & $95\%$ & $100\%$ & $95\%$ & $95\%$ \\
04847 & 855 & $90\%$ & $90\%$ & $95\%$ & $85\%$ \\
05675 & 1120 & $95\%$ & $95\%$ & $100\%$ & $95\%$ \\
05680 & 1051 & $90\%$ & $85\%$ & $85\%$ & $70\%$ \\
05710 & 810 & $95\%$ & $95\%$ & $100\%$ & $90\%$ \\
\hline
Average & 901.67 & $91.67\%$ & $92.50\%$ &  $95.00\%$ & $85.83\%$ \\
Std & 150.87 & $4.08\%$ & $5.24\%$ & $5.48\%$ & $9.70\%$ \\
\hline
\end{tabular}
\label{tab:classifier_accuracy}
\end{table*}

The classifier considered here is the regularized linear discriminant analysis (RLDA). Let $\Sigma$ be the pooled covariance matrix for the two states, $\hat{\mbf\mu}_s=n_s^{-1}\sum_{i \in \text{state } s} \vz_i$ be the sample mean for the state $s \in \{\text{P}, \text{S}\}$, and $n_s$ be the number of time points in state $s$. The RLDA classifier associates a new observation $\vz$ to the label $\hat{s} \in \{\text{P}, \text{S}\}$ according to the Bayes rule
\begin{equation}
\label{eqn:RLDA_Bayes_rule}
\hat{s} = \left\{
\begin{array}{cc}
\text{P}, \quad & \text{if} \quad -(\vz - \bar{\mbf\mu})^\top \Sigma^{-1} \vb + \log(n_\text{S} / n_\text{P}) \le 0 \\
\text{S}, \quad & \text{otherwise}
\end{array}
\right. ,
\end{equation}
where $\bar{\mbf\mu} = (\mbf\mu_\text{P} + \mbf\mu_\text{S}) / 2$ and $\vb = \mbf\mu_\text{P} - \mbf\mu_\text{S}$ where $\mbf\mu_s$ is the mean for the group $s \in \{\text{P}, \text{S}\}$. Note that (\ref{eqn:RLDA_Bayes_rule}) is also equivalent to maximizing the score function
\begin{equation*}
\text{score}(s) = -{1 \over 2} (\vz - \mbf\mu_s)^\top \Sigma^{-1} (\vz - \mbf\mu_s) + \log(n_s / n), \quad n = n_\text{P} +n_\text{S};
\end{equation*}
i.e. $\hat{s} = \text{argmax}_{s \in \{\text{P}, \text{S}\}} \text{score}(s)$. Clearly, $\mbf\mu_s$ and $\Sigma$ are unknown and they need to be estimated from the training data.

We first perform an exploratory data analysis to evaluate the suitability of our method.
As is widely recognized in the neuroscience literature, sparsity is an important feature for the high-solution imaging data. It is well grounded to believe that $\Sigma^{-1}$ and $\Sigma^{-1}\mathbf{b}$ are both sparse (see, for example, \cite{rosa2015sparse,ng2011generalized}).
In addition, we plot the auto-covariance functions (acf) for some voxels. Since S and P states have the block design, we concatenate the blocks with the same label S and P along the time index and make the sample acf plots for each of the two states. Figure~\ref{fig:M161505} shows that some voxels exhibit certain long-memory feature. It has been well-understood that the power spectral density for the fMRI signals has the ``power law" property, suggesting the long-memory behavior of the fMRI time series; see e.g. \cite{bullmoreetal2001a,he2011a}.  Moreover, it is studied in the fMRI literature that the signals reveal light-tail property.  (e.g., they are shown to be sub-Gaussian by \cite{he2011a}).  We further make QQ-norm plots for some voxels (see Figure~\ref{fig:qq_fmri}), which suggests that, the scenario falls into the consideration of Theorem  \ref{thm:linear_stat_rate_bounded}, which allows an ultra high dimension $p$. In our data analysis, the largest number of voxels in all four ROIs is $p=1120$, while the number of time points is $n=320$ and Theorem  \ref{thm:linear_stat_rate_bounded} is apparently suitable.

\begin{figure*}[htbp] 
\centering
      \includegraphics[scale=0.38]{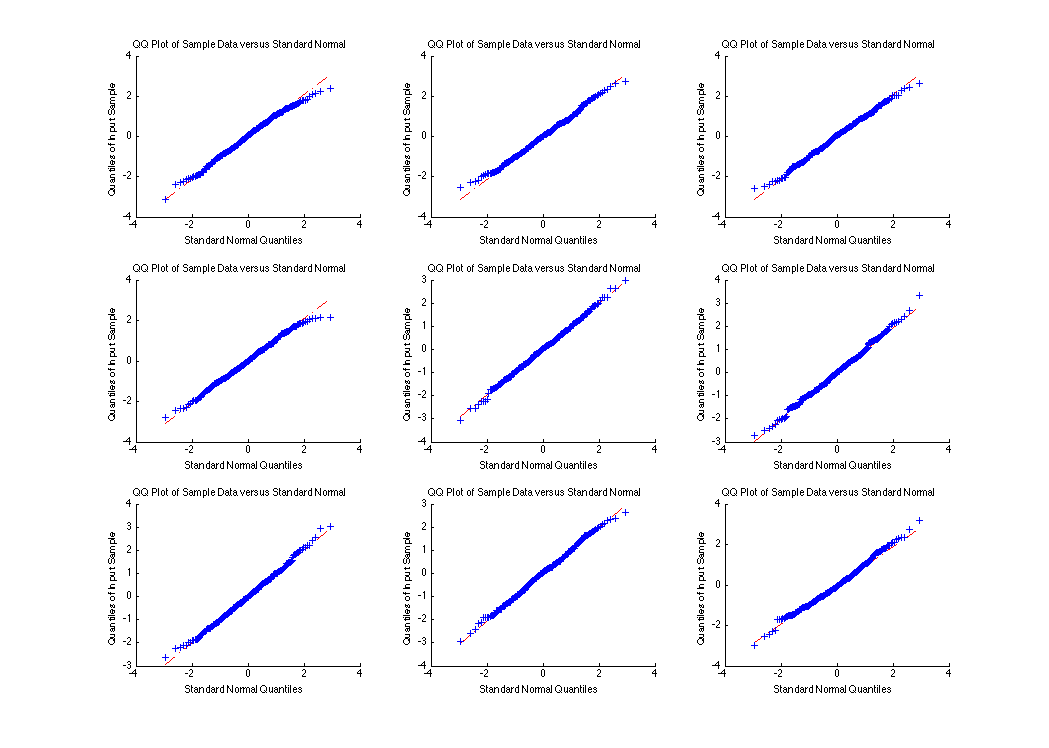}
   \caption{QQ-norm plots of 9 voxels.\label{fig:qq_fmri}}
\end{figure*}

\begin{figure*}[htbp]
\begin{center}
     \includegraphics[width=4.5in]{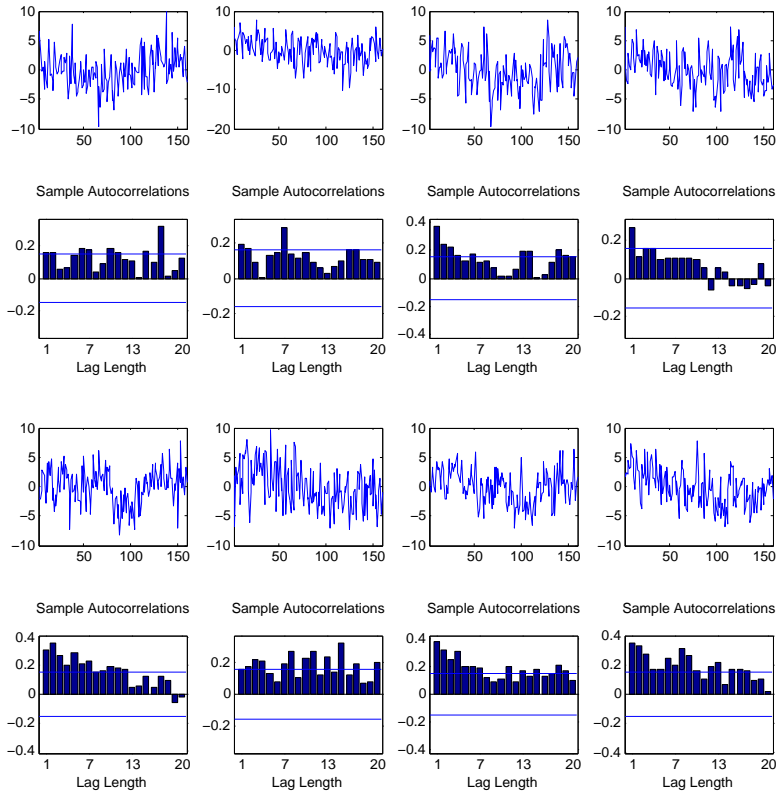}
   \caption{Sample plots for the time series and auto-covariance function of four voxels of the subject 05680. The first and last two rows are from the training data for S and P, respectively.\label{fig:M161505}
}
\end{center}
\end{figure*}

To perform the classification, we use the sample mean estimate $\hat{\mbf\mu}_s$ for $\mbf\mu_s$. Since this fMRI study has a block design meaning that each state lasts for eight consecutive seconds, we average the testing data in eight-second windows as new observations. In our experiment, we take six subjects\footnote{The six subjects are: 04799, 04820, 04847, 05675, 05710 and 05680.} and train an RLDA for each subject. Parameter tuning is performed by the same data splitting procedure used in our simulation studies in Section \ref{sec:simulation}: the first 10 trials used as training dataset (320 time points) and the second 10 trials (320 time points) used as testing dataset. We compare the RLDA with the thresholded sample covariance matrix estimate, precision matrix by the graphical Lasso estimate and linear functional estimate (\ref{eqn:functional}), all plugged into (\ref{eqn:RLDA_Bayes_rule}). Tuning parameters are selected by minimizing the Hamming error on the testing dataset.
We also compare with the performances of the Gaussian Na\"ive Bayes classifier (GNB).\footnote{The LDA is not applicable here since the sample covariance matrix $\hat{S}_n$ on the training data is singular.} The GNB have the same decision rule (\ref{eqn:RLDA_Bayes_rule}) with difference that the diagonal matrix of the sample covariances is used to estimate $\Sigma$. Performances of all classifiers are assessed by the accuracy, which are shown in Table \ref{tab:classifier_accuracy}.

There are interesting observations we can draw from Table \ref{tab:classifier_accuracy}. 
First, we see that, in general, the RLDA classifiers achieve higher accuracy than the GNB classifier. Specifically, accuracy of the RLDA with the three estimates is: $(91.67 \pm 4.08)\%$ for RLDA with the thresholded estimate, $(92.50 \pm 5.24)\%$ for RLDA with the graphical Lasso estimate and $(95.00 \pm 5.48)\%$ for RLDA with linear functional estimate. Accuracy of the GNB is $(85.83 \pm 9.70)\%$. The difference is likely to be explained by the fact that the GBN assumes the independence structure on the covariance matrix $\Sigma$, which is very demanding and potentially can cause serious misspecification problems, as indicated by the lowest accuracy in the classification task. By contrast, the RLDA with the three regularized estimates on $\Sigma^{-1}$ or $\Sigma^{-1} \vb$ is more flexible and it adaptively balances between the bias and variance in the estimation.
Second, among the three RLDA classifiers, we see that the RLDA with direct estimation of the Bayes rule direction $\Sigma^{-1} \vb$ has the highest accuracy, followed by RLDA with the graphical Lasso estimate. As it has been shown in Section \ref{sec:rates_linear_statistics} that, rate of convergence for direct estimation of $\Sigma^{-1} \vb$ can be guaranteed, while it is unclear that whether the consistency of estimating $\Sigma$ or $\Sigma^{-1}$ implies the same property of estimating $\Sigma^{-1} \vb$ with the natural plug-in estimates. In addition, from the scientific viewpoint, it appears to be a meaningful assumption that effective prediction is based on a small number of voxels in the brain since different ROIs may control different tasks and subjects can only perform one task at each time point in the fMRI experiment.

\subsection{Markowitz portfolio allocation}
\label{subsec:markowitz_portfolio}

Here we apply the direct estimation for linear functionals in high-dimensional MP allocation. We use the daily value-weighted returns from January 2005 to March 2015, for 100 portfolios formed on size and the ratio of market equity to book equity, i.e. the intersections of 10 market equity portfolios and 10 of the ratio of book-to-market ratio portfolios.  These portfolios are made using the Center for Research in Security Prices (CRSP) database obtained from the Kenneth French data library. 

The expected return is fixed to $m=1$. At the end of each month from July, 2005 to March, 2015, the portfolios are invested and held for one month with rebalancing. The portfolio allocation weights are estimated using the past 6-month data. Here $p=100$. The sample size $n=129$ approximately as the number of trading days varies slightly from month to month. Four estimators are considered: (1) the linear functional estimator with $\lambda_1$; (2) plug-in estimator using the portfolio daily return mean and the sample covariance matrix from the past data. (We use the Moore-Penrose generalized inverse when the sample covariance matrix is singular;) (3) plug-in estimator using the portfolio daily return means and the graphical lasso precision matrix estimator from the past data; (4) the ridge shrinkage estimator of the covariance matrix by
$$
\text{minimize}_{\vw \in \mathbb{R}^p} \; \vw^\top (\hat{S}_n + \lambda \Id_p) \vw \qquad \text{subject to} \quad \vw^\top \mbf\mu = 1.
$$

The tuning parameters are selected by the following data-driven steps. We partition the data into $K=17$ consecutive periods. Each period consists of $n_{\mathrm{train}}=125$ daily returns as training data and $n_{\mathrm{test}}=21$ daily returns as testing data. The information ratio is computed as
$$
\mathrm{IR}(\lambda)=K^{-1}\sum_{k=1}^K \frac{ \vw_k^\top(\lambda) \mbf\mu_{k,\mathrm{test}}}{\big(\vw_k^\top(\lambda)S_{k,\mathrm{test}}\vw_k(\lambda)\big)^{1/2}},
$$
where $\vw_k(\lambda)$ is the portfolio allocation weight computed using the $k$th period training data with parameter $\lambda$; $\mbf\mu_{k,\mathrm{test}}$ and $S_{k,\mathrm{test}}$ are the sample mean and sample covariance of the $k$th period testing data. The parameters are selected to maximize the information ratio over a grid of $[0,0.1]$, $[0,0.2]$, and $[0,2]$  for the linear functional optimization, graphical lasso and ridge shrinkage, respectively. 

The tuning parameters for the linear functional optimization, graphical lasso and ridge shrinkage are $0.03$, $0.039$ and $1.2$, respectively. Means of the monthly return for the constructed asset portfolios are calculated to represent actual return levels. We also estimated the one-month risk $\vw^\top\hat\Sigma_{\text{one-month}}\vw$ using the estimated weights and the sample covariance of the daily data of the next month. The graphical lasso with parameter $0.039$ has mean return $1.62$ and risk $3.02$, both of which are lower than the other methods. To make the comparison easier, we present the mean return and risk calculated with parameter $0.15$. The result is shown in Table \ref{tab: mp}. It is observed that the linear functional estimator for the Markowitz portfolio allocation performs the best among the four methods in terms of mean return and risk. The performance of the ridge shrinkage is better than the plug-in estimator, but it is worse than the proposed linear functional estimator.

\begin{table}[htp]
\caption{Estimated mean return and risk  of the Fama-French $100$ portfolios analysis.}
\begin{center}
\begin{tabular}{c|cccc}
\hline
         &   Functional & Plug-in & Glasso & Ridge \\
       \hline
Mean Return     &   2.45  &2.00   & 2.38 & 2.37 \\
Risk   &  3.96 &9.08 & 4.17 & 4.57 \\
\hline
\end{tabular}
\end{center}
\label{tab: mp}
\end{table}%

\section*{Acknowledgment}

We thank two anonymous referees for their helpful comments that led to improvements of the paper. X. Chen's research is partially supported by NSF DMS-1404891 and UIUC Research Board Award RB15004. W. B. Wu's research is partially supported by NSF DMS-1405410.

\ifCLASSOPTIONcaptionsoff
  \newpage
\fi



\bibliographystyle{IEEEtran}
\bibliography{cov_GP}

\clearpage

\newpage
\setcounter{page}{1}

\section*{Supplemental material: proofs}
\label{sec:proofs}

In this supplemental material, we prove the main results and technical lemmas of the paper. Equation and reference numbers in the supplemental materials continue from the main paper.

\section{Proofs}
\label{sec:proofs}

\subsection{Proof of Theorem \ref{thm:linear_stat_rate_bounded}--\ref{thm:linear_stat_rate_polynomial}}

Proof of Theorem \ref{thm:linear_stat_rate_bounded}--\ref{thm:linear_stat_rate_polynomial} relies on the following lemma.

\begin{lem}
\label{lem:functional_rate}
Let $\hat\vb$ be an estimator of $\vb$ and $\lambda \ge |\mbf\theta|_1 |\hat{S}_n - \Sigma|_\infty + |\hat\vb-\vb|_\infty$. Then, $\mbf\theta$ satisfies $|\hat{S}_n\mbf\theta-\hat\vb|_\infty\le \lambda$. For the estimate $\hat{\mbf\theta}:=\hat{\mbf\theta}(\lambda)$, we have
\begin{eqnarray}
\label{eqn:functional_wnorm_rate}
|\hat{\mbf\theta} - \mbf\theta|_w \le \left[ 6 D\left(5 \lambda |\Sigma^{-1}|_{L^1} \right) \right]^{1 \over w} \left( 2 \lambda |\Sigma^{-1}|_{L^1}  \right)^{1- {1 \over w}}
\end{eqnarray}
for $1 \le w \le \infty.$
\end{lem}

\begin{proof}
Since $\mbf\theta = \Sigma^{-1} \vb$, we have
\begin{equation*}
\begin{aligned}
|\hat{S}_n\mbf\theta-\hat\vb|_\infty = & |\hat{S}_n\mbf\theta - \Sigma \mbf\theta +\vb-\hat\vb|_\infty \\
\le & |\hat{S}_n-\Sigma|_\infty |\mbf\theta|_1 + |\hat\vb-\vb|_\infty \le \lambda.
\end{aligned}
\end{equation*}
Therefore, $\mbf\theta$ is feasible for (\ref{eqn:functional}) with such choice of $\lambda$ and $|\mbf\theta|_1 \ge |\hat{\mbf\theta}|_1$. Then
\begin{eqnarray*}
|\Sigma \hat{\mbf\theta} - \vb|_\infty &\le& |\Sigma \hat{\mbf\theta} - \hat\vb|_\infty + |\hat\vb-\vb|_\infty \\
&\le& |\hat{S}_n \hat{\mbf\theta} - \hat\vb|_\infty + |\hat{S}_n-\Sigma|_\infty |\hat{\mbf\theta}|_1 + |\hat\vb-\vb|_\infty \\
&\le& 2 \lambda.
\end{eqnarray*}
It follows that $|\hat{\mbf\theta}-\mbf\theta|_\infty \le |\Sigma^{-1}|_{L^1} |\Sigma(\hat{\mbf\theta}-\mbf\theta)|_\infty \le 2 \lambda |\Sigma^{-1}|_{L^1}$. Next, we bound $|\hat{\mbf\theta}-\mbf\theta|_1$. Let $\mbf\delta = \hat{\mbf\theta} - \mbf\theta$ and $u = |\mbf\delta|_\infty$. For any constant $a>1$, let further $\delta^1_j (a) = \hat{\theta}_j \I(|\hat{\theta}_j| \ge au) - \theta_j$ and $\delta^2_j (a)= \delta_j - \delta^1_j(a)$ for $j=1,\cdots,p$. So $\mbf\delta = \mbf\delta^1(a) +\mbf\delta^2(a)$ and
\begin{eqnarray*}
|\mbf\theta|_1 \ge |\hat{\mbf\theta}|_1 &=& |{\mbf\delta}^1(a) + \mbf\theta|_1 + |{\mbf\delta}^2(a)|_1 \\
&\ge& |\mbf\theta|_1 - |{\mbf\delta}^1(a)|_1 + |{\mbf\delta}^2(a)|_1,
\end{eqnarray*}
which implies that $|{\mbf\delta}^1(a)|_1 \ge |{\mbf\delta}^2(a)|_1$ and $|{\mbf\delta}|_1 \le 2 |{\mbf\delta}^1(a)|_1$. Now, observe that
\begin{eqnarray*}
|{\mbf\delta}^1(a)|_1 &=& \sum_j |\hat{\theta}_j \I(|\hat{\theta}_j| \ge au) - \theta_j| \\
&=& \sum_j | \hat{\theta}_j - \theta_j| \I(|\hat{\theta}_j| \ge au)\\
&& +\sum_j |\theta_j| \I(|\hat \theta_j| < au) \\
&\le& u \sum_j \I(|\theta_j| \ge (a-1)u) \\
&& + \sum_j |\theta_j| \I(|\theta_j| \le (a+1)u) \\
&\le& (a-1)^{-1}D\big((a-1)u\big) + D\big((a+1)u\big).
\end{eqnarray*}
Taking $a=1.5$, we obtain 
$$
|\hat{\mbf\theta} - \mbf\theta|_1 \le 6 D\left(5 \lambda|\Sigma^{-1}|_{L^1}  \right).
$$
Now, (\ref{eqn:functional_wnorm_rate}) follows from the interpolation of $\ell^w$ norm by $\ell^\infty$ and $\ell^1$ norms $|\mbf\delta|_w \le |\mbf\delta|_\infty^{1-1/w} |\mbf\delta|_1^{1/w}$.
\end{proof}

Let $G_1 = \{|\hat\vb-\vb|_\infty \le c_b r_b\}$ and $G_2 = \{|\hat{S}_n-\Sigma|_\infty \le C_3 \bar{J}_n\}$, where $\bar{J}_n$ is a sequence of real numbers. Then, $\Prob(G_1) \ge 1-2p^{-C_b}$. In addition, under different dependence and moment conditions, we need to find a $\bar{J}_n$ such that $\Prob(G_2^c) \le C_4 p^{-C_5}$ for some constants $C_4, C_5 > 0$ depending on $C_3$.

\underline{\bf Sub-Gaussian innovations.} Let $x_*=C_3 \max\{(L_{n,\beta} \log{p})^{1/2}, J_{n,\beta} \log{p} \}$. By Lemma \ref{lem:dev-ineq-subGaussian} and the union bound, we have
\begin{eqnarray*}
&& \Prob(|\hat{S}_n - \Sigma|_\infty \ge x_*) \\
&\le& 2 p^2 \max\{ \exp(-C x_*^2 / L_{n,\beta}), \exp(-C x_* / J_{n,\beta}) \} \\
&\le& 2p^2 \max\{ p^{-C C_3^2}, p^{-C C_3} \} = 2p^{-C C_3 +2},
\end{eqnarray*}
where the last step follows from $C_3 > \max(1,2C^{-1})$. Therefore, we can take $C_4=2$ and $C_5 = CC_3-2>0$. By Lemma \ref{lem:functional_rate}, on the event $G_1 \cap G_2$, which occurs with probability at least $1-2p^{-C_b}-C_4p^{-C_5}$, we have
$$|\hat{\mbf\theta} - \mbf\theta|_w \le \left[ 6 D\left(5 \lambda |\Sigma^{-1}|_{L^1} \right) \right]^{1 \over w} \left( 2 \lambda |\Sigma^{-1}|_{L^1}  \right)^{1- {1 \over w}}.$$
Now, (\ref{eqn:linear_stat_wnorm_rate_bounded}) is immediate. Note that (\ref{eqn:linear_stat_wnorm_rate_bounded_Gclass}) easily follows from (\ref{eqn:linear_stat_wnorm_rate_bounded}) in view of $D(u) \le  M_p u^{1-r}$ and $|\mbf\theta|_1 \le \nu^{1-r} M_p$ for $\mbf\theta \in \calG_r(\nu, M_p)$.

\underline{\bf Sub-exponential innovations.} Let $x_*=C_1(\log{p})^{2\alpha+2} n^{-\beta'}$. By Lemma \ref{lem:dev-ineq-exponential}, we have
\begin{eqnarray*}
&& \Prob(|\hat{S}_n - \Sigma|_\infty \ge x_*) \\
&\le& C_2 p^2 \exp(-C C_1^{1/(2\alpha+2)} \log{p}) = C_2 p^{-C' C_1^{1/(2\alpha+2)}+2}.
\end{eqnarray*}
Choose  $C_3 = C' C_1^{1/(2\alpha+2)}-2>0$. Then, Theorem \ref{thm:linear_stat_rate_exponential} follows from Lemma \ref{lem:functional_rate}.

\underline{\bf Polynomial moment innovations.} 
Suppose that $\|\xi_{1,1}\|_q < \infty$. First, consider $\beta \in [1-1/q,\infty)$. By Lemma \ref{lem:dev-ineq-polynomial}, we have for all $x>0$
\begin{equation}
\label{eqn:nagaev-explict-constant-weak}
\Prob(|\hat{S}_n - \Sigma|_\infty \ge x) \le C p^2 [n^{1-q/2} x^{-q/2} + \exp(-C' n x^2)].
\end{equation}
Using $x_\varepsilon = \varepsilon^{-1} p^{4/q} n^{-1+2/q}+C_1(\log(p)/n)^{1/2}$ in the above inequality, we get $\Prob(|\hat{S}_n - \Sigma|_\infty \ge x_\varepsilon) \le C_2 (\varepsilon^{q/2}+p^{-C_3})$, where $C_2 = C$ and $C_3 = C' C_1^2-2$. For $\beta \in (1/2, 1-1/q)$, by Lemma \ref{lem:dev-ineq-polynomial}, we have for all $x>0$
\begin{equation}
\label{eqn:nagaev-explict-constant-strong}
\Prob(|\hat{S}_n - \Sigma|_\infty \ge x) \le C p^2 [n^{1-q/2} x^{-q/2} + n^{-q(2\beta-1)} x^{-q} + \exp(-C' n x^2)].
\end{equation}
Using $x'_\varepsilon = x_\varepsilon+\varepsilon^{-1/2}p^{2/q}n^{1-2\beta}$ in the last inequality, we get $\Prob(|\hat{S}_n - \Sigma|_\infty \ge x'_\varepsilon) \le C_2 (p^{-C_3}+\varepsilon^{q/2})$ with $C_2 = 2C$. 

\qed

\subsection{Proofs of Results in Sections \ref{subsec:MP} and \ref{subsec:sfso_linpred}}

\begin{proof}[Proof of Proposition \ref{prop:mp_property}]
By construction,
$$
{R(\hat\vw) \over R(\vw^*)}= {\Delta_p\hat{\mbf\theta}^\top \Sigma \hat{\mbf\theta} \over \hat\Delta_{p,n}^2} = {\hat{\mbf\theta}^\top \Sigma \hat{\mbf\theta} / \Delta_p  \over ( \bar\vx^\top \hat{\mbf\theta} / \Delta_p)^2 }.
$$
Note that $\check{S}_n=\hat{S}_n+U_n$ where $U_n=(\bar\vx-\mbf\mu)(\bar\vx-\mbf\mu)^\top$. With our choice of $\lambda$, by Lemma \ref{lem:functional_rate} and {\bf MP 1, 2}, $|\hat{\mbf\theta}|_1$  is bounded in probability by $|\mbf\theta|_1$. By {\bf MP 1, 2, 3}, and {\bf 4}, we have
\begin{eqnarray*}
|\hat{\mbf\theta}^\top \Sigma \hat{\mbf\theta} - \mbf\theta^\top \Sigma \mbf\theta| &\le& |\hat{\mbf\theta}^\top (\Sigma\hat{\mbf\theta}-\hat S_n\mbf\theta)| + |(\hat{\mbf\theta}^\top\hat{S}_n-\mbf\theta^\top\Sigma)\mbf\theta| \\
&\le& |\hat{\mbf\theta}^\top \Sigma (\hat{\mbf\theta}-\mbf\theta)| + |(\hat{\mbf\theta}-\mbf\theta)^\top \Sigma \mbf\theta| \\
&& \quad + 2 |\hat{\mbf\theta}^\top (\hat{S}_n-\Sigma) \mbf\theta| \\
&\le& |\Sigma|_\infty (|\hat{\mbf\theta}|_1 + |\mbf\theta|_1) |\hat{\mbf\theta}-\mbf\theta|_1 \\
&& \quad + 2 ( |\check{S}_n-\Sigma|_\infty + |\bar\vx-\mbf\mu|_\infty^2 ) |\hat{\mbf\theta}|_1 |\mbf\theta|_1 \\
&\lesssim_\Prob&  r_1 |\mbf\theta|_1 + (r_2+r_3^2) |\mbf\theta|_1^2.
\end{eqnarray*}
Be aware that $r_1$ depends on $\lambda$. Since $|\mbf\theta|_1 = O(\Delta_p s)$, we have
$$
\begin{aligned}
\left| {\hat{\mbf\theta}^\top \Sigma \hat{\mbf\theta} \over \Delta_p} - 1 \right| \lesssim_\Prob & {r_1 |\mbf\theta|_1 + (r_2+r_3^2) |\mbf\theta|_1^2 \over \Delta_p } \\
=& O\left(s r_1 + \Delta_p s^2 (r_2+r_3^2) \right).
\end{aligned}
$$
Similarly,
\begin{eqnarray*}
|\hat{\mbf\theta}^\top \bar\vx - \mbf\theta^\top \mbf\mu| &\le& |\hat{\mbf\theta}^\top (\bar\vx-\mbf\mu)| + |(\hat{\mbf\theta}^\top-\mbf\theta)^\top \mbf\mu| \\
&\le& |\hat{\mbf\theta}|_1 |\bar\vx-\mbf\mu|_\infty + |\hat{\mbf\theta}-\mbf\theta|_1 |\mbf\mu|_\infty \\
&=& O_\Prob(\Delta_p s r_3 + r_1).
\end{eqnarray*}
Therefore,
$$
\left| {\hat{\mbf\theta}^\top \bar\vx \over \Delta_p} - 1 \right| = O_\Prob(s r_3 + {r_1 \over \Delta_p}).
$$
By {\bf MP 3}, $\Delta_p \ge m^2 / C$. If $s r_1 + \Delta_p s^2  (r_2+r_3^2) = o(1)$, then the result follows from the continuous mapping theorem.
\end{proof}

\begin{proof}[Proof of Proposition \ref{prop:sfso_linpred_property}]

By the decomposition in \cite[Theorem 2]{chen2015}, we have
\begin{eqnarray*}
|\hat\Gamma_n-\Gamma|_\infty &\le& T + n^{-1} \max_{1\le s\le \lfloor c l\rfloor} s |\gamma_s| + \max_{l<s\le n-1} |\gamma_s|,
\end{eqnarray*}
where
$$
T = n^{-1} \max_{0 \le s\le \lfloor c l\rfloor}\left|{\sum_{i=1}^{n-s}X_iX_{i+s}-\E X_iX_{i+s}}\right|.
$$
Since $|a_m| \le C_0 m^{-\beta}$ for $m \ge 1$, by Lemma \ref{lem:summability}, $r_s = O(s^{-\beta})$ and $O(s^{1-2\beta})$ for $\beta > 1$ and $1>\beta>1/2$, resp. Therefore, we have $\max_{1\le s\le \lfloor c l\rfloor} s|\gamma_s| = O(1)$ or $O(l^{2(1-\beta)})$ if $\beta > 1$ or $1>\beta>1/2$; and $\max_{l<s\le n-1} |\gamma_s| = O(l^{-\beta})$ or $O(l^{1-2\beta})$ if $\beta>1$ or $1>\beta>1/2$. By Lemma \ref{lem:bound_max_autocovariances}, $T=O_\Prob(r_5)$. Then (\ref{eqn:sfso_rate}) follows from \ref{lem:functional_rate}. The ratio consistency of $\hat{\mbf\theta}$ follows from the assumption that $ R(\mbf\theta) \ge C>0$ and 
\begin{eqnarray*}
R(\hat{\mbf\theta}) - R(\mbf\theta) &=& \mbf\gamma^\top \Gamma^{-1} \mbf\gamma + \hat{\mbf\theta}^\top \Gamma \hat{\mbf\theta} - 2 \hat{\mbf\theta}^\top \mbf\gamma \\
&=& \mbf\theta^\top \Gamma \mbf\theta + \hat{\mbf\theta}^\top \Gamma \hat{\mbf\theta} - 2 \hat{\mbf\theta}^\top \Gamma \mbf\theta \\
&=& (\mbf\theta-\hat{\mbf\theta})^\top \Gamma (\mbf\theta-\hat{\mbf\theta}) \cr
&\le& K_1 |\hat{\mbf\theta}-\mbf\theta|_1^2.
\end{eqnarray*}
\end{proof}

\section{Technical lemmas}

In this section, we prove the technical lemmas that are used for Section \ref{sec:proofs} in this supplemental material.

\begin{lem}
\label{lem:summability}
Let $\beta > 1/2$ and $(a_m)_{m \in \mathbb{Z}}$ be a real sequence such that $a_m \le C_0m^{-\beta}$ for $m \ge 1$ and $a_m = 0$ if $m < 0$. Let $\gamma_k = \sum_{m=0}^\infty |a_m a_{m+k}|$, $\theta_k = |a_k| A_{k+1}$, where $A_k = (\sum_{l=k}^\infty a_l^2)^{1/2}$, $\delta_n = \sum_{i=-n}^\infty (\sum_{k=i+1}^{i+n} \theta_k )^2$. Let $b_{s,m} = \sum_{i=1}^{n-s} a_{i-m} a_{i+s-m}$ and $b_{s,m,m'} = \sum_{i=1}^{n-s} a_{i-m} a_{i+s-m'} + a_{i-m'} a_{i+s-m}$. Then (i) $\gamma_n = O(n^{-\beta})$ (resp. $O(n^{-1} \log n)$, or $O(n^{1-2\beta})$) and $\sum_{k=0}^n \gamma_k = O(1)$ (resp. $O(\log^2{n})$, or $O(n^{2-2\beta})$) hold for $\beta > 1$ (resp. $\beta = 1$, or $1 > \beta > 1/2$); (ii) $\theta_n = O(n^{-2\beta+1/2})$; (iii) $\sum_{k=0}^n \gamma_k^2 = O(1)$ (resp. $O(\log n)$, or $O(n^{3-4\beta})$) and $\delta_n  = O(n)$ (resp. $O(n \log^2{n})$, or $O(n^{4-4\beta})$) for $\beta > 3/4$ (resp. $\beta = 3/4$, or $3/4 > \beta > 1/2$); (iv) $\sum_{m \in \mathbb{Z}} \max_{0 \le s < n} b_{s,m}^2 = O(n)$; (v) for $q \ge 2$, $\sum_{m'<m} \max_{0 \le s < n} |b_{s,m,m'}|^q = O(n)$ (resp. $O(n \log{n})$, or $O(n^{2+(1-2\beta)q})$) for $\beta>1/2+1/(2q)$ (resp. $\beta=1/2+1/(2q)$, or $1/2+1/(2q)>\beta>1/2$); (vi) $\sum_{m \in \mathbb{Z}} \max_{0 \le s < n} (\sum_{m'<m} b_{s,m,m'}^2)^2$ $=O(n)$ (resp. $O(n \log^2{n})$, or $O(n^{7-8\beta})$) for $\beta>3/4$ (resp. $\beta=3/4$, or $3/4>\beta>1/2$). The constants of $O(\cdot)$ only depend on $C_0$ and $\beta$ for (i)-(iv) and (vi), and they may also depend on $q$ for (v).
\end{lem}

Lemma \ref{lem:summability} follows from elementary manipulations. The details are omitted.  In Lemma \ref{lem:dev-ineq-subGaussian}, \ref{lem:dev-ineq-polynomial}, and \ref{lem:dev-ineq-exponential}, we assume that the linear process has mean-zero and $\hat{S}_n=n^{-1} \sum_{i=1}^n \vx_i \vx_i^\top$.

\begin{lem}[Sub-Gaussian]
\label{lem:dev-ineq-subGaussian}
Let $(\xi_{i,j})$ be i.i.d. satisfying (\ref{eq:J280946p}). Assume (\ref{eqn:coefs_decay}). Then for all $x > 0$
\begin{equation}
\label{eqn:dev-ineq-subGaussian}
\Prob(|\hat{s}_{jk}-\sigma_{jk}| \ge x) \le 2 \exp \left[ -C \min\left( {x^2 \over L_{n,\beta}}, \; {x \over J_{n,\beta}} \right) \right],
\end{equation}
where $(L_{n,\beta}, J_{n,\beta}) = (n^{-1}, n^{-1})$, $(n^{-1}, n^{1-2\beta})$ and $(n^{2-4\beta}, n^{1-2\beta})$ for $\beta > 1$, $1 > \beta > 3/4$ and $3/4 > \beta > 1/2$, respectively, and $C$ is a constant that only depend on $\beta$, $C_0$ in (\ref{eqn:coefs_decay}) and $C_\xi$ in (\ref{eq:J280946p}).
\end{lem}

\begin{proof}
Let $\veta = (\mbf\xi_n^\top, \mbf\xi_{n-1}^\top, \ldots)^\top$ and
\begin{equation*}
A^{(j)} = \left(
\begin{array}{cccccccc}
A_{0,j\cdot} & A_{1,j\cdot}  & A_{2,j\cdot}  & \cdots & A_{n-1,j\cdot}  & A_{n,j\cdot}  & \cdots &   \\
0 & A_{0,j\cdot} & A_{1,j\cdot} & \cdots & A_{n-2,j\cdot} &  A_{n-1,j\cdot} & \cdots & \\
0 & 0 & A_{0,j\cdot} & \cdots & A_{n-3,j\cdot} &  A_{n-2,j\cdot} & \cdots &  \\
\vdots & \vdots & \vdots & \ddots & \vdots & \vdots & \vdots \\
0 & 0 & 0 & \cdots & A_{0,j\cdot} &  A_{1,j\cdot} & \cdots & \\
\end{array} \right).
\end{equation*}
Observe that $(X_{n,j},\cdots,X_{1,j})^\top =  A^{(j)} \veta.$ Then $n\hat{s}_{jk} = \veta^\top\rbr{A^{(j)}}^\top A^{(k)}\veta$. Since $\xi_{i,j}$ are i.i.d. sub-Gaussian, by the Hanson-Wright inequality  \cite[Theorem 1.1]{rudelsonvershynin2013a}, 
\begin{eqnarray}
\label{eqn:hw_subgaussian_1}
&&\Prob\left( \left|\veta^\top(A^{(j)})^\top A^{(k)}\veta - \E(\veta^\top(A^{(j)})^\top A^{(k)}\veta) \right| \ge x \right) \\ \nonumber
&\le& 2 \exp\rBr {-C \min\rbR {\rBR{(A^{(j)})^\top A^{(k)}}_F^{-2}x^2, \; {\rho\rbr{(A^{(j)})^\top A^{(k)}}} ^{-1}x}},
\end{eqnarray}
where $C$ is a constant independent of $p$, $n$ and $x$. Let $\Gamma^{(j)} = A^{(j)} (A^{(j)})^\top$. Then, $\Gamma^{(j)}$ has the same set of nonzero real eigenvalues as $ (A^{(j)})^\top A^{(j)}$. Since 
\begin{equation*}
\left| (A^{(j)})^\top A^{(k)} \right|_F^2 = \tr\left[ A^{(j)} {(A^{(j)})}^\top  A^{(k)} {(A^{(k)})}^\top \right] \le \left| \Gamma^{(j)} \right|_F \left| \Gamma^{(k)} \right|_F
\end{equation*}
and 
\begin{equation*}
\rho[{(A^{(j)})^\top A^{(k)}}]\le \rho({A^{(j)}})\rho({A^{(k)}}) = \rho(\Gamma^{(j)})^{1/2} \rho(\Gamma^{(k)})^{1/2},
\end{equation*}
the right-hand side of (\ref{eqn:hw_subgaussian_1}) is bounded by
\begin{equation}
\label{eqn:hw_subgaussian_2}
\le 2 \exp\left[ -C \min\left( {x^2 \over \max_{j \le p} |\Gamma^{(j)}|_F^2}, \; {x \over \max_{j \le p} \rho(\Gamma^{(j)})} \right) \right].
\end{equation}

By the Cauchy-Schwarz inequality, we have
\begin{eqnarray*}
\gamma_l^{(j)} := \sum_{m=0}^\infty |A_{m,j\cdot} A_{m+l,j\cdot}^\top|  \le \sum_{m=0}^\infty \left(\sum_{k=1}^p a_{m,jk}^2 \right)^{1/2} \left(\sum_{k=1}^p a_{m+l,jk}^2 \right)^{1/2}.
\end{eqnarray*}
By the decay condition (\ref{eqn:coefs_decay}) and Lemma \ref{lem:summability} (i), we have $\gamma_l^{(j)} = O(l^{-\beta})$ if $\beta > 1$ and $\gamma_l^{(j)} = O(l^{1-2\beta})$ if $1 > \beta > 1/2$ uniformly over $j$. Here and hereafter in this proof, the constants of $O(\cdot)$ only depend on $C_0$ and $\beta$. Also by Lemma \ref{lem:summability}, $|\Gamma^{(j)}|_F^2 \le n{\gamma_0^{(j)}}^2 + 2 \sum_{l=1}^{n-1} (n-l) {\gamma_l^{(j)}}^2 \le 2 n \sum_{l=0}^{n-1} {\gamma_l^{(j)}}^2$, which is of order $O(n)$ or $O(n^{4-4\beta})$ for $ \beta > 3/4$ or $3/4 > \beta > 1/2$, respectively. Similarly, since $\rho(\Gamma^{(j)}) \le 2 \sum_{l=0}^n \gamma_l^{(j)} = O(1)$ or $O(n^{2-2\beta})$ for $\beta > 1$ or $1 > \beta > 1/2$, respectively. Now, (\ref{eqn:dev-ineq-subGaussian}) follows from Lemma \ref{lem:summability} and (\ref{eqn:hw_subgaussian_2}).
\end{proof}

In the following Lemma \ref{lem:dev-ineq-polynomial} and \ref{lem:dev-ineq-exponential}, without loss of generality, we may consider the mean-zero linear process $X_i:=X_{i1}=\sum_{m=0}^\infty \va_m \mbf\xi_{i-m}$, where $\va_m$ is the first row of $A_m$ such that in accordance to (\ref{eqn:coefs_decay}), $|\va_m|\le C_0(1\vee m)^{-\beta}$, for $m\ge 0,\beta>1/2,$ and $\va_m=\vzero$ if $m<0$. Let $\hat{S}_n= n^{-1}\sum_{i=1}^n X_i^2$ and $\sigma^2 = \E X_i^2$.

\begin{lem}[Polynomial moment]
\label{lem:dev-ineq-polynomial}
Let $q > 2$ and $(\xi_{i,j})$ be i.i.d. random variables such that $\|\xi_{i,j}\|_{2q}<\infty$. Assume (\ref{eqn:coefs_decay}) holds. Let $\mu_{0,q}=\max(\|\xi_{1,1}^2-1\|_q^q, \|\xi_{1,1}\|_{q}^{2q})$. Then (i) If $\beta \ge 1-1/(2q)$, then we have for all $x>0$
\begin{equation}
\label{eqn:poly_beta>1-1/(4q)}
\Prob(|\hat{S}_n-\sigma^2| \ge x) \le C \left\{ {\mu_{0,q} \vee \|\xi_{1,1}\|_{2q}^{4q} \over n^{q-1} x^q}  + \exp\left(-{C‘ n x^2 \over \mu_{0,2} \vee \|\xi_{1,1}\|_2^4} \right) \right\}.
\end{equation}
(ii) If $1-1/(2q) > \beta > 1/2$, then
\begin{equation}
\label{eqn:poly_1-1/(4q)>beta>3/4}
\Prob(|\hat{S}_n-\sigma^2| \ge x) \le C \left\{ {\mu_{0,q} \over n^{q-1} x^q} +  {\|\xi_{1,1}\|_{2q}^{4q} \over n^{2q(2\beta-1)} x^{2q}} + \exp\left(-{C’ n x^2 \over \mu_{0,2}} \right) \right\},
\end{equation}
where the constants $C$ and $C'$ only depend on $q$,$\beta$ and $C_0$.
\end{lem}

\begin{proof}
In this proof, the constants $C,C_1,\cdots$ and the constants in $O(\cdot)$ only depend on $C_0$, $\beta$ and $q$. 
Let $Q_n=\sum_{i=1}^n W_i$, where
$$W_i=\sum_{m\in \Z} \sum_{m'=m+1}^\infty \va_m \mbf\xi_{i-m} \mbf\xi_{i-m'}^\top \va_{m'}^\top =\sum_{m\in \Z} \sum_{m'=-\infty}^{m-1} \va_{i-m} \mbf\xi_m \mbf\xi_{m'}^\top \va_{i-m'}^\top.$$
Let $Z_m=\mbf\xi_m \mbf\xi_m^\top -\Id_p$ be iid random matrices in $\mathbb{R}^{p \times p}$. Write $\hat{S}_n = L_n + 2 Q_n$, where
$$L_n=\sum_{i=1}^n \sum_{m \in \Z} \va_{i-m} Z_m \va_{i-m}^\top = \sum_{m\in\Z} \tr( Z_m B_m ),  \quad B_m=\sum_{i=1}^n \va_{i-m}^\top \va_{i-m}.$$
Since $\tr( Z_m B_m ), m \in \mathbb{Z}$ are independent random variables with mean zero, by Corollary 1.7 in \cite{nagaev1979a}, we have for all $x>0$
\begin{equation*}
\Prob(|L_n| \ge x) \le C_1 {\sum_{m\in\Z} \E|\tr( Z_m B_m )|^q \over x^q} + 2 \exp\left(-{C_2 x^2 \over \sum_{m\in\Z}  \E|\tr( Z_m B_m )|^2}\right).
\end{equation*}
Note that
$$
\E|\tr( Z_m B_m )|^q \le C^{q-1} \left[ \E\left|\sum_{s=1}^p Z_{m,ss} B_{m,ss}\right|^q + \E\left|\sum_{s=1}^p \sum_{t<s} Z_{m,st} B_{m,st}\right|^q \right].
$$
Since $(\xi_{m,s} \sum_{t<s} B_{m,st} \xi_{m,t})_{s=1,\cdots,p}$ is a martingale difference sequence w.r.t. $\calF_s^m = \sigma(\xi_{m,1},\cdots,\xi_{m,s})$, we have by Burkholder's inequality \cite{rio2009a}
\begin{eqnarray}
\label{eqn:polynomial_linear_1}
\|\sum_{s=1}^p Z_{m,ss} B_{m,ss}\|_q^2 &\le& (q-1) \sum_{s=1}^p B_{m,ss}^2 \|\xi_{0,0}^2-1\|_q^2, \\
\label{eqn:polynomial_linear_2}
\|\sum_{s=1}^p \sum_{t<s} Z_{m,st} B_{m,st}\|_q^2 &\le& (q-1)^2 \sum_{s=1}^p \sum_{t<s} B_{m,st}^2 \|\xi_{0,0}\|_q^4.
\end{eqnarray}
Therefore, it follows that $\E|\tr( Z_m B_m )|^q \lesssim \mu_{0,q} |B_m|_F^q$. By the Cauchy-Schwarz inequality and (\ref{eqn:coefs_decay}), we have
$$
|B_m|_F \le \sum_{i=1}^n |\va_{i-m}|^2 = O(\sum_{i=1}^n (i-m)^{-2\beta}) \quad \text{if } i \ge m,
$$
and $|B_m|_F=0$ if $i < m$. Simple calculations show that, e.g. see the proof of Theorem 1 in \cite{wu2016}, $\sum_{m\in\mathbb{Z}}|B_m|_F^q=O(n)$ for $q\ge2$ and $\beta>1/2$. Therefore, we have
\begin{equation}
\label{eqn:poly_linear_term}
\Prob(|L_n| \ge x) \le C_1 {n \mu_{0,q} \over x^q} + 2 \exp\left(-{C_2 x^2 \over n \mu_{0,2}}\right).
\end{equation}
Next, we deal with $Q_n$. Let $W_{i,j}=\E(W_i | \mbf\xi_{i-j},\cdots, \mbf\xi_i)$, $D_{i,j}=W_{i,j}-W_{i,j-1}$ and $Q_{i,j}=\sum_{k=1}^i W_{k,j}$. Let $0=\tau_0<\tau_1<\cdots<\tau_L=n$ be a subsequence of $\{1,\cdots,n\}$, where $\tau_l=2^l, 1\le l \le L-1$ and $L=\lfloor \log_2{n} \rfloor$. Since $Q_{n,0}=\sum_{i=1}^n W_{i,0}=0$, we have the decomposition
$$
Q_n = Q_n -Q_{n,n} + \sum_{l=1}^L (Q_{n,\tau_l}-Q_{n,\tau_{l-1}}).
$$
For each $j \ge 0$, we have $D_{i,j}=\va_j \mbf\xi_{i-j} \sum_{m=i-j+1}^i \va_{i-m} \mbf\xi_m$ and
$$
\cP_k D_{i,j} = \left\{
\begin{array}{cc}
\va_j \mbf\xi_{i-j} \va_{i-k} \mbf\xi_k & \text{if } i-j+1 \le k \le i \\
0 & \text{otherwise}
\end{array} \right. ,
$$
where $\calP_k(\cdot)=\E(\cdot | \mbf\xi_k,\mbf\xi_{k-1},\cdots)-\E(\cdot | \mbf\xi_{k-1},\mbf\xi_{k-2},\cdots)$ is the projection operator on $\mbf\xi_k$. By Burkholder's inequality, we have
\begin{eqnarray*}
\| Q_n - Q_{n,n}\|_{2q}^2 &\le& (2q-1) \sum_{k=-\infty}^n \left\|\sum_{j=n+1}^\infty \sum_{i=1}^n \cP_k D_{i,j} \right\|_{2q}^2 \\
&=& (2q-1) \sum_{k=-\infty}^n \left\|\sum_{j=n+1}^\infty \sum_{i=1}^n \va_j \mbf\xi_{i-j} \va_{i-k} \mbf\xi_k \vone_{(k \le i \le k+j-1)} \right\|_{2q}^2 \\
&=& (2q-1) \sum_{k=-\infty}^n \left\| \sum_{j=n+1}^\infty \mbf\xi_k^\top \sum_{m=1-j}^{n-j} \va_{m+j-k}^\top \va_j \mbf\xi_m \vone_{(k-j \le m \le k-1)} \right\|_{2q}^2.
\end{eqnarray*}
By Fubini's theorem,
$$
\sum_{j=n+1}^\infty \sum_{m=1-j}^{n-j} = \sum_{m=-n}^{-1} \sum_{j=n+1}^{n-m} + \sum_{m=-\infty}^{-n-1} \sum_{j=1-m}^{n-m}.
$$
Thus, we get $\| Q_n - Q_{n,n}\|_{2q} \lesssim (T_1+T_2)^{1/2},$ where
\begin{eqnarray*}
T_1 &=& \sum_{k=-\infty}^n \left\| \sum_{m=-n}^{-1} \mbf\xi_k^\top B_{1mk} \mbf\xi_m \vone_{(m \le k-1)} \right\|_{2q}^2, \quad B_{1mk}=\sum_{j=n+1}^{n-m} \va_{m+j-k}^\top \va_j \vone_{(j \ge k-m)}, \\
T_2 &=& \sum_{k=-\infty}^n \left\| \sum_{m=-\infty}^{-n-1} \mbf\xi_k^\top B_{2mk} \mbf\xi_m \vone_{(m \le k-1)} \right\|_{2q}^2, \quad B_{2mk} = \sum_{j=1-m}^{n-m}\va_{m+j-k}^\top \va_j \vone_{(j \ge k-m)}.
\end{eqnarray*}
First, we tackle $T_2$. For $i=1,2,$ observe that $(\mbf\xi_k^\top B_{imk} \mbf\xi_m)_{m=\cdots,k-2,k-1}$ are backward martingale differences w.r.t. $\sigma(\mbf\xi_m,\cdots,\mbf\xi_k)$. Using Burkholder's inequality twice and by the Cauchy-Schwarz inequality, we have
\begin{eqnarray*}
T_2 &\le& (2q-1) \sum_{k=-\infty}^n \sum_{m=-\infty}^{-n-1} \|\mbf\xi_k^\top B_{2mk} \mbf\xi_m \vone_{(m \le k-1)}\|_{2q}^2 \\
&\lesssim& (2q-1)^2 \|\xi_{0,0}\|_{2q}^4 \sum_{k=-\infty}^n \sum_{m=-\infty}^{-n-1} |B_{2mk}|_F^2 \vone_{(m \le k-1)} \\
&\le& (2q-1)^2 \|\xi_{0,0}\|_{2q}^4 \sum_{k=-\infty}^n \sum_{m=-\infty}^{-n-1}\left( \sum_{j=1-m}^{n-m} |\va_{m+j-k}| \cdot |\va_j| \vone_{(j \ge k-m)} \right)^2.
\end{eqnarray*}
Therefore, by (\ref{eqn:coefs_decay}),
\begin{eqnarray*}
{T_2 \over (2q-1)^2 \|\xi_{0,0}\|_{2q}^4} &\lesssim& \sum_{k=-\infty}^n \sum_{m=-\infty}^{-n-1} \left( \sum_{j=1-m}^{n-m} j^{-\beta} [j-(k-m)+1]^{-\beta} \right)^2 \vone_{(m \le k-1)} \\
&\lesssim& \sum_{k=-\infty}^{-n} \sum_{m=-\infty}^{k-1} n^2 (1-m)^{-2\beta} (1-k)^{-2\beta} \\
&& + \sum_{k=-n+1}^0 \sum_{m=-\infty}^{-n-1} (1-m)^{-2\beta} (\sum_{j=1}^n (j-k)^{-\beta})^2 \\
&& + \sum_{k=1}^n \sum_{m=-\infty}^{-n-1} (k-m)^{-2\beta} (\sum_{j=k}^n (j-k+1)^{-\beta} )^2.
\end{eqnarray*}
By Karamata's theorem and some elementary manipulations, we have
$$
T_2 = \left\{
\begin{array}{cc}
O(\|\xi_{1,1}\|_{2q}^4 n^{2-2\beta}) & \text{if } \beta>1 \\
O(\log^3 (n))& \text{if } \beta=1\\
O(\|\xi_{1,1}\|_{2q}^4 n^{4-4\beta}) & \text{if } 1>\beta>1/2
\end{array} \right. .
$$
For $T_1$, we apply a similar argument and it obeys the same bound as in $T_2$. Therefore, we have
$$
\| Q_n - Q_{n,n}\|_{2q} = \left\{
\begin{array}{cc}
O(\|\xi_{1,1}\|_{2q}^2 n^{1-\beta}) & \text{if } \beta>1 \\
O(\log^{3/2} (n))& \text{if } \beta=1\\
O(\|\xi_{1,1}\|_{2q}^2 n^{2-2\beta}) & \text{if } 1>\beta>1/2
\end{array} \right. 
$$
and by Markov's inequality
\begin{equation*}
\label{eqn:poly_quad_term_remainder}
\Prob(|Q_n-Q_{n,n}| \ge x) \le {\E|Q_n-Q_{n,n}|^{2q} \over x^{2q}} =
\left\{
\begin{array}{cc}
O(\|\xi_{1,1}\|_{2q}^{4q} n^{2(1-\beta)q} x^{-2q} ) & \text{if } \beta>1 \\
O(\|\xi_{1,1}\|_{2q}^{4q} \log^{3q} (n) x^{-2q} ) & \text{if } \beta= 1 \\
O(\|\xi_{1,1}\|_{2q}^{4q} n^{4(1-\beta)q} x^{-2q}) & \text{if } 1>\beta>1/2
\end{array} \right. .
\end{equation*}
Now, we deal with $Q_{n,\tau_l}-Q_{n,\tau_{l-1}}$. Fix an $l = 1,\cdots,L$ and let $\bar{r} = \lceil n/\tau_l \rceil$ and $B_r = \{1+(r-1) \tau_l, \cdots, (r \tau_l) \wedge n\}$ be the $r$-th block of $\{1,\cdots,n\}$ for $1 \le r \le \bar{r}$. Let
$$
Y_{l,r} = \sum_{j=\tau_{l-1}+1}^{\tau_l} \sum_{i \in B_r} D_{i,j}.
$$
Since $D_{i,j}$ is $j$-dependent for all $i$, it follows that $Y_{l,1}, Y_{l,3},\cdots$ are independent and so are $Y_{l,2}, Y_{l,4},\cdots$. Let $\lambda_l=(6/\pi^2) l^{-2}, 1 \le l \le L$. So $\sum_{l=1}^L \lambda_l \le 1$. By Corollary 1.7 in \cite{nagaev1979a}, we have
\begin{eqnarray*}
\Prob(|Q_{n,\tau_l}-Q_{n,\tau_{l-1}}| \ge 2 \lambda_l x) &\le& C_1 {\sum_{r=1}^{\bar{r}} \| Y_{l,r} \|_{2q}^{2q} \over \lambda_l^{2q} x^{2q}} + 4 \exp\left(-{C_2 \lambda_l^2 x^2 \over \sum_{r=1}^{\bar{r}} \| Y_{l,r} \|_2^2}\right).
\end{eqnarray*}
We need to bound $\|Y_{l,r}\|_{2q}^{2q}$. It suffices to consider the first block $r=1$. By a similar argument as in bounding $\| Q_n - Q_{n,n}\|_{2q}^2$, we have by Fubini's theorem $\|Y_{l,r}\|_{2q} = O((T_3+T_4)^{1/2})$, where
\begin{eqnarray*}
T_3 &=& \sum_{k=-\infty}^{\tau_l} \left\| \sum_{m=0}^{\tau_l-\tau_{l-1}-1} \mbf\xi_k^\top B_{3mk} \mbf\xi_m \vone_{(m \le k-1)} \right\|_{2q}^2, \quad B_{3mk}= \sum_{j=\tau_{l-1}+1}^{\tau_l-m}  \va_{m+j-k}^\top \va_j \vone_{(j \ge k-m)}, \\
T_4 &=& \sum_{k=-\infty}^{\tau_l} \left\| \sum_{m=-\tau_{l-1}}^{-1} \mbf\xi_k^\top B_{4mk} \mbf\xi_m \vone_{(m \le k-1)} \right\|_{2q}^2, \quad B_{4mk} = \sum_{j=\tau_{l-1}+1}^{\tau_l} \va_{m+j-k}^\top \va_j \vone_{(j \ge k-m)}.
\end{eqnarray*}

By Burkholder's inequality and Karamata's theorem, we get
\begin{eqnarray*}
{T_3 \over (2q-1)^2 \|\xi_{0,0}\|_{2q}^4} &\le& \sum_{k=1}^{\tau_l} \sum_{m=0}^{\tau_l-\tau_{l-1}-1} \left( \sum_{j=\tau_{l-1}+1}^{\tau_l-m} |\va_j| \cdot |\va_{m+j-k}| \vone_{(j \ge k-m\ge 1)} \right)^2 \\
&=& \left\{
\begin{array}{cc}
O(\tau_{l-1}^{-2\beta} \tau_l^2) & \text{if } \beta>1 \\
O(\log ^2(\tau_{l})) & \text{if } \beta=1 \\
O(\tau_{l-1}^{-2\beta} \tau_l^{4-2\beta}) & \text{if } 1>\beta>1/2
\end{array} \right. 
\end{eqnarray*}
and
\begin{eqnarray*}
{T_4 \over (2q-1)^2 \|\xi_{0,0}\|_{2q}^4} &\le& \sum_{k=-\tau_{l-1}+1}^{\tau_l} \sum_{m=-\tau_{l-1}}^{-1} \left( \sum_{j=\tau_{l-1}+1}^{\tau_l}  |\va_j| \cdot |\va_{m+j-k}| \vone_{(j \ge k-m\ge 1)} \right)^2 \\
&=& \left\{
\begin{array}{cc}
O(\tau_{l-1} \tau_l^{3-4\beta}) & \text{if } \beta>1 \\
O(\log ^2(\tau_{l})) & \text{if } \beta=1 \\
O(\tau_{l-1}^{1-2\beta} \tau_l^{3-2\beta}) & \text{if } 1>\beta>1/2
\end{array} \right. .
\end{eqnarray*}
Since $T_4=O(T_3)$, we have: if $\beta>1$, then
\begin{eqnarray*}
\Prob(|Q_{n,\tau_l}-Q_{n,\tau_{l-1}}| \ge 2 \lambda_l x) &\le& C_1 { n \tau_l^{-1} \| \xi_{1,1} \|_{2q}^{4q} (\tau_{l-1}^{-\beta} \tau_l)^{2q} \over \lambda_l^{2q} x^{2q}} \\
&& + 4 \exp\left(-{C_2 \lambda_l^2 x^2 \over  n \tau_l^{-1} \| \xi_{1,1} \|_2^4 (\tau_{l-1}^{-\beta} \tau_l)^2}\right);
\end{eqnarray*}
if $1>\beta>1/2$
\begin{eqnarray*}
\Prob(|Q_{n,\tau_l}-Q_{n,\tau_{l-1}}| \ge 2 \lambda_l x) &\le& C_1 { n \tau_l^{-1} \| \xi_{1,1} \|_{2q}^{4q} (\tau_{l-1}^{-\beta} \tau_l^{2-\beta})^{2q} \over \lambda_l^{2q} x^{2q}} \\
&& + 4 \exp\left(-{C_2 \lambda_l^2 x^2 \over  n \tau_l^{-1} \| \xi_{1,1} \|_2^4 (\tau_{l-1}^{-\beta} \tau_l^{2-\beta})^2}\right).
\end{eqnarray*}

{\bf Case I: $\beta>1$}. We have
\begin{eqnarray*}
\Prob(|Q_n| \ge 3 x) &\le& C_1 {n^{2(1-\beta)q} \|\xi_{1,1}\|_{2q}^{4q} \over x^{2q}} + C_2 { n \| \xi_{1,1} \|_{2q}^{4q} \over x^{2q}} \sum_{l=1}^L {\tau_l^{-1} (\tau_{l-1}^{-\beta} \tau_l)^{2q} \over \lambda_l^{2q} } \\
&& + \min\left\{ 4 \sum_{l=1}^L \exp\left(-{C_3  x^2 \over \| \xi_{1,1} \|_2^4 n} {\lambda_l^2 \over \tau_l^{-1} (\tau_{l-1}^{-\beta} \tau_l)^2}\right), \; 1 \right\}.
\end{eqnarray*}
For $l \ge 1$, with the choice of $\tau_l$ and $\lambda_l$, we have
$$
{\lambda_l^2 \tau_{l-1}^{2\beta} \over \tau_l} = \left({6 \over \pi^2}\right)^2 {2^{2\beta(l-1)} \over l^4 2^l} = {36 \over 4^\beta \pi^4} 2^{(2\beta-1)l - 4\log_2{l}} \ge \phi_1 > 0
$$
and
$$
\sum_{l=1}^L {\tau_l^{2q-1} \over \tau_{l-1}^{2q\beta} \lambda_l^{2q}} \lesssim \sum_{l=1}^{\log_2{n}} 2^{(2q-1-2q\beta)l} l^{4q} < \infty
$$
because $2q-1-2q\beta<-1$. Therefore, 
$$
 \min\left\{ 4 \sum_{l=1}^L \exp\left(-{C_3  x^2 \over \| \xi_{1,1} \|_2^4 n} {\lambda_l^2 \over \tau_l^{-1} (\tau_{l-1}^{-\beta} \tau_l)^2}\right), \; 1 \right\} \le C_1  \exp\left(-{C_2  x^2 \over \| \xi_{1,1} \|_2^4 n} \right).
$$
Hence, we obtain that 
\begin{equation*}
\Prob(|\hat{S}_n-\sigma^2| \ge x) \le C_1 \left\{ {\mu_{0,q} \over n^{q-1} x^q} +  {\|\xi_{1,1}\|_{2q}^{4q} \over n^{2q-1} x^{2q}} + \exp\left(-{C_2 n x^2 \over \mu_{0,2} \vee \|\xi_{1,1}\|_2^4} \right) \right\}.
\end{equation*}
Now, we may assume that $x \ge C_q n^{-1+1/q}$ for some constant $C_q$, because otherwise the inequality (\ref{eqn:poly_beta>1-1/(4q)}) is trivial. Then, $n^{1-2q} x^{-2q} \le C_q n^{1-q} x^{-q}$, from which (\ref{eqn:poly_beta>1-1/(4q)}) follows.

{\bf Case II: $\beta=1$}. We have
\begin{eqnarray*}
\Prob(|Q_n| \ge 3 x) &\le& C_1 {\log ^{3q}(n) \|\xi_{1,1}\|_{2q}^{4q} \over x^{2q}} + C_2 { n \| \xi_{1,1} \|_{2q}^{4q} \over x^{2q}} \sum_{l=1}^L {\tau_l^{-1} (\log^{2q}(\tau_l)) \over \lambda_l^{2q} } \\
&& + \min\left\{ 4 \sum_{l=1}^L \exp\left(-{C_3  x^2 \over \| \xi_{1,1} \|_2^4 n} {\lambda_l^2 \over \tau_l^{-1} \log^2(\tau_l)}\right), \; 1 \right\}.
\end{eqnarray*}
By similar argument as in Case I, we obtain (\ref{eqn:poly_beta>1-1/(4q)}).

{\bf Case III:} The long-memory case with $1>\beta>1/2$. We have
\begin{eqnarray*}
\Prob(|Q_n| \ge 3 x) &\le& C_1 {\|\xi_{1,1}\|_{2q}^{4q} n^{4(1-\beta)q} \over x^{2q}} + C_2 { n \| \xi_{1,1} \|_{2q}^{4q} \over x^{2q}} \sum_{l=1}^L {\tau_l^{-1} (\tau_{l-1}^{-\beta} \tau_l^{2-\beta})^{2q} \over \lambda_l^{2q} } \\
&& + \min\left\{ 4 \sum_{l=1}^L \exp\left(-{C_3  x^2 \over \| \xi_{1,1} \|_2^4 n} {\lambda_l^2 \over \tau_l^{-1} (\tau_{l-1}^{-\beta} \tau_l^{2-\beta})^2}\right), \; 1 \right\}.
\end{eqnarray*}
For $l \ge 1$, we have 
\begin{enumerate}
\item[(i)] If $3/4<\beta<1$,
$$
{\lambda_l^2 \tau_{l-1}^{2\beta} \over \tau_l^{3-2\beta}} = \left({6 \over \pi^2}\right)^2 {2^{2\beta(l-1)} \over l^4 2^{(3-2\beta)l}} = {36 \over 4^\beta \pi^4} 2^{(4\beta-3)l - 4\log_2{l}} \ge \phi_2 > 0
$$
and
\begin{eqnarray*}
\sum_{l=1}^L {\tau_l^{2q(2-\beta)-1} \over \tau_{l-1}^{2q\beta} \lambda_l^{2q}} &\lesssim& \sum_{l=1}^{\log_2{n}} 2^{[4(1-\beta)q-1]l} l^{4q} \\
&=& \left\{
\begin{array}{cc}
O(1) & \text{if } 1>\beta>1-1/(4q) \\
O(\log^{4q+1} (n)) & \text{if } \beta=1-1/(4q) \\
O(n^{4(1-\beta)q-1}) & \text{if } 1-1/(4q)>\beta>3/4
\end{array} \right. .
\end{eqnarray*}
Hence, we obtain that 
\begin{equation*}
\Prob(|\hat{S}_n-\sigma^2| \ge x) \le C_1 \left\{ {\mu_{0,q} \over n^{q-1} x^q} +  n^{4(1-\beta)q}{\|\xi_{1,1}\|_{2q}^{4q} \over n^{2q} x^{2q}} + \exp\left(-{C_2 n x^2 \over \mu_{0,2} \vee \|\xi_{1,1}\|_2^4} \right) \right\}.
\end{equation*}
We may assume that $x \ge C_q n^{-1+1/q}$ for some constant $C_q$, because otherwise the inequality (\ref{eqn:poly_beta>1-1/(4q)}) is trivial. Then, if $\beta\ge1-1/(2q)$, $n^{2q(1-2\beta)} x^{-2q} \le C_q n^{1-q} x^{-q}$, and then (\ref{eqn:poly_beta>1-1/(4q)}) follows.

\item[(ii)] If $3/4\ge \beta>1/2$, then by a similar argument for proving the bounds on $T_1$ and $T_2$ terms, we can show that $\|Q_{n,n}\|_{2q}$ obeys the same bound as $\|Q_n-Q_{n,n}\|_{2q}$, i.e. $\|Q_{n,n}\|_{2q} = O(\|\xi_{1,1}\|_{2q}^2 n^{2(1-\beta)})$. By Markov's inequality,
$$
\Prob(|Q_n| \ge x) \le C_1 {n^{4(1-\beta)q} \|\xi_{1,1}\|_{2q}^{4q} \over x^{2q}}.
$$
Combining this with (\ref{eqn:poly_linear_term}), we have (\ref{eqn:poly_1-1/(4q)>beta>3/4}).
\end{enumerate}
\end{proof}

\begin{lem}[Sub-exponential]
\label{lem:dev-ineq-exponential}
Assume $(\xi_{i,j})$ are i.i.d. random variables satisfying (\ref{eq:J290945}), $\alpha > 1/2$. Let $\beta'=\min(1/2, 2\beta-1)$ for $\beta>1/2$. Then we have for all $x>0$
\begin{equation}
\label{eqn:dev-ineq-exponential_LRD_>1/2}
\Prob(|\hat{S}_n-\sigma^2| \ge x) \le C \exp\left[-C' (n^{\beta'} x)^{1 \over 2 \alpha + 2} \right],
\end{equation}
where the constants $C$ and $C'$ only depend on $\alpha$,$\beta$,$C_0$ in (\ref{eqn:coefs_decay}) and $C_{\xi,\alpha}$ in (\ref{eq:J290945}).
\end{lem}

\begin{proof}
First, consider the quadratic component $Q_n = \sum_{i=1}^n W_i$. Let $\theta_k = |\va_k| A_{k+1}$ and $A_k^2 = \sum_{m=k}^\infty |\va_m|^2$ for $k \ge 0$. Put $\theta_k = 0$ if $k < 0$. By Lemma \ref{lem:summability}, $\theta_k \le C_\beta k^{-2\beta+1/2}$. Note that $\calP_k Q_n, k = \cdots, n-1, n$, are martingale differences. Since $\calP_0 W_i = \va_i \mbf\xi_0 \sum_{m=1}^\infty \va_{i+m} \mbf\xi_{-m}$, we have by \cite[Theorem 1(i)]{MR2353389}, Burkholder's inequality \cite{rio2009a}, and Lemma \ref{lem:summability}
\begin{eqnarray}
\nonumber
\|Q_n\|_q^2 &\le& (q-1) \sum_{i=-n}^\infty \left( \sum_{k=i+1}^{i+n} \|\calP_0 W_k\|_q \right)^2 \\
\nonumber
&\le& (q-1)^2 \sum_{i=-n}^\infty \left[ \sum_{k=i+1}^{i+n} \|\va_k \mbf\xi_0 \|_q \left( \sum_{m=1}^\infty \|\va_{k+m} \mbf\xi_{-m}\|_q^2\right)^{1/2} \right]^2 \\
\label{eqn:Q_n_strong_LRD_bounded}
&\le& C q^{4 \alpha + 4} \sum_{i=-n}^\infty \left( \sum_{k=i+1}^{i+n} \theta_k \right)^2 \le C q^{4\alpha + 4} U_n^2,
\end{eqnarray}
where $U_n = n^{1/2}$ if $\beta > 3/4$ and $U_n = n^{2-2\beta}$ if $3/4 > \beta > 1/2$. Therefore, $\|Q_n\|_q \le C q^{2 \alpha + 2} U_n$ for $q \ge 2$. Let $\lambda = 1/(2\alpha+2)$. By Stirling's formula, we have
\begin{equation*}
\limsup_{q \to \infty} {t \|U_n^{-1} Q_n\|_{\lambda q}^\lambda \over (q!)^{1/q} } \le \limsup_{q \to \infty} {e t C^\lambda \lambda q \over q (2 \pi q)^{1 / (2q)} } = e \lambda t C^\lambda < 1,
\end{equation*}
for $0 < t < (e\lambda C^\lambda)^{-1}$. Thus, for sufficiently large $q_0 = q_0(\alpha)$, $\sum_{q=q_0}^\infty {t^q \|U_n^{-1} Q_n\|_{\lambda q}^{\lambda q} / q!} < \infty$. By the exponential Markov inequality and Taylor's expansion $e^v = \sum_{q=0}^\infty v^q / q!$, we have
$$
\Prob(Q_n \ge x) \le \exp(-tx^\lambda / U_n^\lambda) \E \exp[t|U_n^{-1} Q_n|^\lambda] \le C  \exp(-tx^\lambda / U_n^\lambda).
$$
The linear component follows from similar lines with the difference that $\|\tr(Z_m B_m)\|_q^2 = O(q^{4\alpha+2} |B_m|_F^2)$; see (\ref{eqn:polynomial_linear_1}) and (\ref{eqn:polynomial_linear_2}). Therefore, we get
\begin{equation*}
\Prob(|\hat{S}_n-\sigma^2| \ge x) \le C \exp\left[-C' \min\left( (n^{\beta'} x)^{1 \over 2 \alpha + 2},  (n^{1/2} x)^{2 \over 4 \alpha + 3}  \right) \right].
\end{equation*}
Assume that $n^{\beta'} x \ge 1$ because otherwise we can choose $C$ large enough to make (\ref{eqn:dev-ineq-exponential_LRD_>1/2}) trivially hold. Then,$(n^{1/2} x)^{2/(4\alpha+3)} \ge (n^{\beta'} x)^{2/(4\alpha+3)} \ge (n^{\beta'} x)^{1/(2\alpha+2)}$ and (\ref{eqn:dev-ineq-exponential_LRD_>1/2}) follows.
\end{proof}

Next, we prove a maximal inequality for the auto-covariances of a univariate linear process.
\begin{lem}
\label{lem:bound_max_autocovariances}
Suppose that $X_i$ is a univariate linear time series (\ref{eq:one_dim_ts}) such that $\|\xi_0\|_q<\infty, q \ge 4$. Let $1 < J < n$ and
$$T = n^{-1} \max_{0 \le s \le J} |\sum_{i=1}^{n-s} (X_i X_{i+s} - \E(X_i X_{i+s}))|.$$
Then we have
\begin{equation}
\label{eqn:bound_max_autocovariances}
T = O_\Prob( (\log{J}) n^{-\beta'} \|\xi_0\|_q^2), \quad \text{where } \beta'=\min(1/2, 2\beta-1) \text{ for } \beta \neq 3/4.
\end{equation}
\end{lem}


\begin{proof}
Let $L_s = \sum_{m \in \mathbb{Z}} b_{s,m} (\xi_m^2-1)$, where $b_{s,m}=\sum_{i=1}^{n-s} a_{i-m} a_{i+s-m}$. By \cite[Lemma 8]{cck2014b}, we have
\begin{eqnarray*}
\E \max_{0 \le s \le J} |L_s| &\lesssim& \|\xi_0^2-1\|  \left(\max_{0 \le s \le J} \sum_{m \in \mathbb{Z}} b_{s,m}^2 \right)^{1/2} \sqrt{\log{J}} + \left( \E[\max_{0 \le s \le J} \max_{m \in \mathbb{Z}} b_{s,m}^2 (\xi_m^2-1)^2 ] \right)^{1/2} \log{J} \\
&\le&  \|\xi_0^2-1\| \left[  \left(\max_{0 \le s \le J} \sum_{m \in \mathbb{Z}} b_{s,m}^2 \right)^{1/2} \sqrt{\log{J}}  + \left( \sum_{m \in \mathbb{Z}}  \max_{0 \le s \le J} b_{s,m}^2 \right)^{1/2} \log{J} \right] \\
&\lesssim& \|\xi_0^2-1\| \left( \sum_{m \in \mathbb{Z}}  \max_{0 \le s \le J} b_{s,m}^2 \right)^{1/2} \log{J}.
\end{eqnarray*}
By Lemma \ref{lem:summability} and Markov's inequality, we have
\begin{equation}
\label{eqn:linear_term_maximal_ineq_autocov}
\max_{0 \le s \le J} |L_s| = O_\Prob(\|\xi_0^2-1\| n^{1/2} \log{J}).
\end{equation}

Let $b_{s,m,m'} = \sum_{i=1}^{n-s} a_{i-m} a_{i+s-m'} + a_{i-m'} a_{i+s-m}$ and consider $Q_s = \sum_{m \in \mathbb{Z}} \sum_{m'<m} b_{s,m,m'} \xi_m \xi_{m'}$. By the randomization inequality \cite[Theorem 3.5.3]{delaPenaGine1999},
$$
\E(\max_{0\le s \le J} |Q_s|) \lesssim \E \max_{0\le s \le J} \left| \sum_{m' < m} \varepsilon_m \varepsilon_{m'} b_{s,m,m'} \xi_m \xi_{m'} \right|,
$$
where $\varepsilon_m$'s are i.i.d. Rademacher random variables independent of $\xi_m$'s. Let the triangle matrix $\Xi = (b_{s,m,m'} \xi_m \xi_{m'})_{m'<m}$. Since $\varepsilon_m$'s are sub-Gaussian, by the Hanson-Wright inequality \cite[Theorem 1.1]{rudelsonvershynin2013a} conditionally on $\mbf\xi=(\xi_m)_{m \in \mathbb{Z}}$, we have
$$
\Prob(| \sum_{m' < m} \varepsilon_m \varepsilon_{m'} b_{s,m,m'} \xi_m \xi_{m'} | \ge t \mid \mbf\xi) \le 2 \exp\left[-C \min\left({t^2 \over |\Xi|_F^2}, {t \over \rho(\Xi)}\right)\right].
$$
Then, it follows from integration-by-parts and Pisier's inequality \cite[Lemma 2.2.2]{vandervaartwellner1996} that
$$
\E(\max_{0\le s \le J} |Q_s|) \lesssim (\log{J}) \sqrt{I}, \quad \text{where } I=\E(\max_{0\le s \le J} \sum_{m'<m} b_{s,m,m'}^2 \xi_m^2 \xi_{m'}^2).
$$
By the triangle inequality,
\begin{equation}
\label{eqn:quad_ineq}
I \lesssim \max_{0\le s \le J} \sum_{m'<m} b_{s,m,m'}^2 + II + III,
\end{equation}
where 
\begin{eqnarray*}
II &=& \E\left[ \max_{0\le s \le J} \left| \sum_{m'<m} b_{s,m,m'}^2 (\xi_m^2 - 1) (\xi_{m'}^2-1) \right| \right], \\
III &=& \E \left[ \max_{0\le s \le J} \left| \sum_{m'<m} b_{s,m,m'}^2 (\xi_m^2 - 1) \right| \right].
\end{eqnarray*}
Since $\sum_{m'<m} b_{s,m,m'}^2 (\xi_m^2 - 1) (\xi_{m'}^2 - 1)$ is a completely degenerate $U$-statistic, by the randomization inequality \cite[Theorem 3.5.3]{delaPenaGine1999}, the above argument and the Jensen and Cauchy-Schwarz inequalities, we obtain that
\begin{eqnarray*}
II \lesssim \E\left[ \max_{0\le s \le J} \sum_{m'<m} \varepsilon_m \varepsilon_{m'} b_{s,m,m'}^2 \xi_m^2 \xi_{m'}^2 \right] &\lesssim& (\log{J}) \E \max_{0\le s \le J} \left[\sum_{m'<m} b_{s,m,m'}^4 \xi_m^4 \xi_{m'}^4 \right]^{1/2} \\
&\le&  (\log{J}) \sqrt{\E B} \sqrt{I},
\end{eqnarray*}
where $B= \max_{0\le s \le J} \max_{m'<m} b_{s,m,m'}^2 \xi_m^2 \xi_{m'}^2$. By \cite[Lemma 8]{cck2014b}, we have
\begin{eqnarray*}
&& \E \left[ \max_{0\le s \le J}  \left| \sum_{m'<m} b_{s,m,m'}^2 (\xi_m^2 - 1) \right| \right] \\
&\lesssim& \sqrt{\log{J}} \max_{0 \le s \le J} \left[\sum_{m \in \mathbb{Z}} (\sum_{m'<m} b_{s,m,m'}^2)^2\right]^{1/2} \Var^{1/2}(\xi_0^2) + (\log{J}) \sqrt{\E B'},
\end{eqnarray*}
where $B' = \max_{0\le s \le J} \max_{m \in \mathbb{Z}} (\sum_{m'<m} b_{s,m,m'}^2)^2 (\xi_m^2 - 1)^2$. Now, solving the quadratic inequality (\ref{eqn:quad_ineq}), we have
\begin{eqnarray*}
I &\lesssim& (\log{J})^2 \E B + \max_{0\le s \le J} \sum_{m'<m} b_{s,m,m'}^2 \\
&& + (\log{J}) \left[ \sum_{m \in \mathbb{Z}} \max_{0 \le s \le J} \left( \sum_{m'<m} b_{s,m,m'}^2 \right)^2 \right]^{1/2} \Var^{1/2}(\xi_0^2).
\end{eqnarray*}
By Lemma \ref{lem:summability},
\begin{eqnarray*}
\E B \le \|B\|_{q/2} &\le& (\sum_{m'<m} \max_{0\le s \le J} |b_{s,m,m'}|^q )^{2\over q} \|\xi_0\|_{q}^4 = 
\left\{
\begin{array}{cc}
O(n^{2\over q} \|\xi_0\|_{q}^4) & \text{if } \beta > {1\over2} + {1\over 2q} \\
O(n^{2\over q} (\log{n})^{2\over q} \|\xi_0\|_{q}^4) & \text{if } \beta = {1\over2} + {1\over 2q} \\
O(n^{{4\over q} + 2(1-2\beta)} \|\xi_0\|_{q}^4) & \text{if } {1\over2} + {1\over 2q} > \beta > {1\over2} \\
\end{array}
\right. ,
\end{eqnarray*}
$$
\max_{0\le s \le J} \sum_{m'<m} b_{s,m,m'}^2 = \left\{
\begin{array}{cc}
O(n) & \text{if } \beta > 3/4 \\
O(n^{4-4\beta}) & \text{if } 3/4 > \beta > 1/2 \\
\end{array}
\right. ,
$$
and
$$
\sum_{m \in \mathbb{Z}} \max_{0 \le s \le J} \left( \sum_{m'<m} b_{s,m,m'}^2 \right)^2 = \left\{
\begin{array}{cc}
O(n) & \text{if } \beta > 3/4 \\
O(n^{7-8\beta}) & \text{if } 3/4 > \beta > 1/2 \\
\end{array}
\right. .
$$
Since $\log{J} = O(\log{n})$ and $q \ge 4$, it follows that
\begin{eqnarray*}
\E(\max_{0\le s \le J} |Q_s|)  
&\lesssim& \left\{
\begin{array}{cc}
(\log{J}) n^{1/2} \|\xi_0\|_q^2 & \text{if } \beta > 3/4 \\
(\log{J}) n^{2-2\beta} \|\xi_0\|_q^2 & \text{if } 3/4 > \beta > 1/2 \\
\end{array}
\right. .
\end{eqnarray*}
Combining this with (\ref{eqn:linear_term_maximal_ineq_autocov}), we have (\ref{eqn:bound_max_autocovariances}).
\end{proof}

\end{document}